\pdfoutput=1 
\documentclass[a4paper, 10pt]{amsart}

\usepackage[protrusion=true,expansion=true]{microtype}
\usepackage{amsthm,amsfonts,amsmath,amssymb}
\usepackage{pxfonts}
\usepackage{euscript}
\usepackage[utf8]{inputenc}
\usepackage{bbold,bbm}
\usepackage{hyperref}
\hypersetup{
    colorlinks,
    linkcolor={red!50!black},
    citecolor={blue!50!black},
    urlcolor={blue!80!black}
}
\usepackage{graphics}
\usepackage{epstopdf} 
\usepackage{xypic}
\usepackage{float}
\usepackage{graphicx}
\usepackage{comment}
\usepackage[most]{tcolorbox}
\usepackage{mathtools} 

\newtheorem{thm}{Theorem}[section]
\newtheorem{lem}[thm]{Lemma}
\newtheorem{prop}[thm]{Proposition}
\newtheorem{cor}[thm]{Corollary}

\newtheorem{prop-defn}[thm]{Proposition-Definition}

\theoremstyle{definition}
\newtheorem{defn}[thm]{Definition}

\newtheorem{rem}[thm]{Remark}

\newtheorem{notat}[thm]{Notation}
\theoremstyle{remark}

\numberwithin{equation}{section}
\numberwithin{thm}{section}

\newcommand{\be}{\begin{equation}}
\newcommand{\ee}{\end{equation}}
\newcommand{\bnu}{\begin{enumerate}}
\newcommand{\enu}{\end{enumerate}}

\newcommand{\C}{\mathbb{C}}

\newcommand\CA           {\EuScript{A}}

\newcommand\CC           {\EuScript{C}}

\newcommand\CI         {\EuScript{I}}

\newcommand\CL          {\EuScript{L}}
\newcommand\CM          {\EuScript{M}}
\newcommand\CN         {\EuScript{N}}

\newcommand\CP         {\EuScript{P}}

\newcommand\CR         {\EuScript{R}}

\newcommand{\bB}{\mathbf{B}}

\newcommand{\vN}{\mathbf{vN}}

 \DeclareMathOperator{\End}{End}

 \DeclareMathOperator{\bdd}{bdd}

 \DeclareMathOperator{\BMod}{BiMod}
 \DeclareMathOperator{\sBMod}{s-BiMod}

 \DeclarePairedDelimiterX\braket[2]{\langle}{\rangle}{#1 \delimsize\vert #2}

\newcommand{\one}{\mathbb{1}}
\newcommand{\zero}{\mathbb{0}}

\newcommand{\dfeq}{\vcentcolon=}
\newcommand{\mref}[1]{(\ref{#1})}
\newcommand{\rt}[1]{\underset{#1}{\otimes}}

\newif\ifFinal
\Finalfalse
\ifFinal
    \excludecomment{hide}
\else
    \includecomment{hide}
\fi

\begin{document}
\title[Realization of rigid C$^*$-bicategories as bimodules]{Realization of rigid C$^*$-bicategories as bimodules\\ over type II$_1$ von Neumann algebras}
\author{Luca Giorgetti$^{1,2}$, Wei Yuan$^{3,4}$}

\address{1 Department of Mathematics, Vanderbilt University \\
1326 Stevenson Center, Nashville, TN 37240, US}

\address{2 Dipartimento di Matematica, Universit\`a di Roma Tor Vergata\\
Via della Ricerca Scientifica, 1, I-00133 Roma, Italy}
\email{giorgett@mat.uniroma2.it}

\address{3 Institute of Mathematics, Academy of Mathematics and Systems Science \\
Chinese Academy of Sciences, Beijing, 100190, China}

\address{4 School of Mathematical Sciences, University of Chinese Academy of Sciences \\
Beijing 100049, China}
\email{wyuan@math.ac.cn}

\begin{abstract}
We prove that every rigid C$^*$-bicategory with finite-dimensional centers (finitely decomposable horizontal units) can be realized as Connes' bimodules over finite direct sums of II$_1$ factors. In particular, we realize every multitensor C$^*$-category as bimodules over a finite direct sum of II$_1$ factors.
\end{abstract}

\subjclass[2010]{} 
\thanks{Research by the first author (L.G.) is supported by the European Union's Horizon 2020 research and innovation programme under Grant Agreement 795151 ``beyondRCFT'' H2020-MSCA-IF-2017. Research by the second author (W.Y.) is supported by the NSFC under grant numbers 11971463, 11871303, 11871127.}

\maketitle
\tableofcontents

\section{Introduction}

C$^*$-bicategories are a generalization of monoidal C$^*$-categories \cite{glr, dr, lro} to the higher categorical world \cite{lei, la, jy}. 
A bicategory (also called 2-category if we consider only strict ones) can be depicted as nodes on a two-dimensional plane (the objects), oriented links or arrows between them (the 1-morphisms), and arrows between arrows with common source and target nodes (the 2-morphisms). Both 1-morphisms and 2-morphisms can be concatenated giving rise to a pictorial representation of two composition products, respectively called horizontal and vertical composition. A bicategory is called a C$^*$-bicategory when the 2-morphisms are complex vector spaces endowed with an antilinear orientation-reversing involution and a Banach space norm, which are compatible with the composition and which fulfil the C$^*$ identity $\|f^*f\| = \|f\|^2$. 
Assuming the C$^*$-structure is maybe the most restrictive hypothesis that we make in this paper. 
It is motivated by the applications to quantum physics (quantum field theory \cite{bklr}, quantum information \cite{longl}, quantum computation \cite{wan}) and operator algebras \cite{lands}, where the role of 2-morphisms is usually played by bounded operators on a Hilbert space. Going down in the tower of generalizations, a C$^*$-bicategory with only one object is a monoidal C$^*$-category, a (monoidal) C$^*$-category with only one object is an (abelian) unital C$^*$-algebra.

In a bicategory, one always assumes the existence of horizontal and vertical units for the respective composition products. For 2-morphisms, the notion of inverse is the ordinary algebraic one: two 2-morphisms are inverse to one another if their composition on either side is the respective unit 2-morphism. For 1-morphisms, the ordinary notion of inversion is given by asking that either composition is isomorphic to the respective unit 1-morphism, rather than equal to it, and it is often replaced by the weaker notion of left/right conjugation \cite{lro} (also called dualization \cite{saa} or adjunction \cite{ml} in the special case of functors). Two 1-morphisms are dual to one another if their compositions contain the respective unit 1-morphism, and the containment relation is given by a pair of 2-morphisms solving the so-called conjugate equations (also sometimes called duality equations or zig-zag equations). In a C$^*$-bicategory, 2-morphisms can be unitarily invertible and there is no distinction between left and right dualizability for 1-morphisms. A second assumption that we make in this paper, which is more customary in the literature, is the existence of a dual for every 1-morphism. We refer to this property as rigidity. In concrete examples of tensor C$^*$-categories and C$^*$-bicategories coming from operator algebras, it is related to the finiteness of the Jones index \cite{vfr} for the objects or 1-morphisms in question.

Conversely, it is a natural and in various guises widely studied question to ask, whether every such abstract C$^*$-categorical structure can be represented as endomorphisms or bimodules over C$^*$-algebras or von Neumann algebras, with finite Jones index, and what is the type of algebras that may arise in this way \cite{p5, pa, delphi, pu, ps}. There are several results in the literature \cite{hy, y0, fv, bhp, fr, lw} providing realizations in the case of rigid monoidal C$^*$-categories with simple (also called irreducible) unit. The latter will be referred to as tensor C$^*$-categories in the following.

The purpose of the present work is to extend these realization results in two directions: on the one hand we allow for finitely decomposable tensor units (e.g., multifusion or multitensor C$^*$-categories), on the other hand we allow for bicategories with more than one object (thus providing a realization for rigid C$^*$-bicategories). 

We actually do not give a particular realization for every such structure, but we produce a machinery that lifts any realization technique from tensor C$^*$-categories to rigid C$^*$-bicategories with finitely decomposable horizontal units.
This finite decomposability assumption, that we call finite-dimensionality of the centers of the objects in the rigid C$^*$-bicategory, is equivalent to the finite-dimensionality of the 2-morphisms vector spaces.
We need to assume it in order to have a well-behaved theory of intrinsic (matrix) dimension for abstract rigid C$^*$-bicategories, and a canonical choice of the so-called standard solutions of the conjugate equations with certain optimization, scalarity and functoriality properties.
The von Neumann algebras we obtain, on which the C$^*$-bicategory is realized as bimodules, are indeed finite direct sums of type II$_1$ factors.
We deal only with small categories and bicategories. Thanks to a previous work \cite{lw}, we do not have to add any cardinality constraint on the spectrum of the categories (the set of unitary equivalence classes of simple objects or simple 1-morphisms) that one can realize. 

The paper is organized as follows. In the first part of Section \ref{sec:Cstarbicats}, we review some basics of rigid C$^*$-bicategories and multitensor C$^*$-categories. 

In Section \ref{sec:stdsol}, we recall the definition of standard solution of the conjugate equations, extend it to non-connected 1-morphisms and show that the tensor product of standard solutions is again a standard solution in the case of tensor products of a 1-morphism with its dual.

In Section \ref{sec:specialalgs}, we review the notion of (C$^*$-Frobenius) algebra and bimodule in a C$^*$-bicategory. We introduce special C$^*$-Frobenius algebras and special bimodules and we show that every C$^*$-Frobenius algebra is isomorphic to a special one, generalizing a result of \cite{bklr}. Similarly, every bimodule over a special C$^*$-Frobenius algebra is isomorphic to a special one.

In Section \ref{sec:specialbimods}, we show that special C$^*$-Frobenius algebras and special bimodules in a rigid C$^*$-bicategory form a C$^*$-bicategory with respect to a relative horizontal composition of bimodules. This new C$^*$-bicategory of special bimodules has the same properties of the former C$^*$-bicategory (rigidity, finite-dimensionality of the centers), and it contains it as a full sub-C$^*$-bicategory.   

In Section \ref{sec:multisbimcat}, which is the main technical part of the paper, we introduce standard C$^*$-Frobenius algebras and we show that every indecomposable multitensor C$^*$-category can be realized as a special bimodule category over a fixed standard C$^*$-Frobenius algebra living in a tensor C$^*$-category. In this sense, multitensor C$^*$-categories are finite-dimensional amplifications of tensor C$^*$-categories.

In Section \ref{sec:connes_bim}, we review the theory of Connes' bimodules over finite direct sums of II$_1$ factors, namely we spell out the properties of the rigid C$^*$-bicategory of finite bimodules over type II$_1$ von Neumann algebras with finite-dimensional centers.

In Section \ref{sec:qsystemsinvN}, we provide a proof of the fact that a C$^*$-Frobenius algebra in the multitensor C$^*$-category of bimodules over a fixed type II$_1$ von Neumann algebra with finite-dimensional center gives an extension of that algebra, generalizing the analogous result in the factor case. Moreover, we prove our realization result for indecomposable multitensor C$^*$-categories, from which a corollary on the realizability of multifusion C$^*$-categories over finite direct sums of the hyperfinite II$_1$ factor immediately follows. 

Section \ref{sec:pfmainthm} is dedicated to the proof of the main result:

\begin{thm}\label{thm:real_2_cat}
Every rigid C$^*$-bicategory with finite-dimensional centers can be realized as finite bimodules over finite direct sums of II$_1$ factors.
\end{thm}

A natural question is if arbitrary rigid C$^*$/W$^*$-bicategories can be realized as finite bimodules over von Neumann algebras, dropping the finite-dimensionality assumption on the centers, thus on the Banach spaces of 2-morphisms, that we made in this note. 

Another (presumably much harder) long-standing open question is about the realizability over hyperfinite von Neumann algebras of countably generated (rigid) tensor and multitensor C$^*$-categories.

\section{Rigid C$^*$-bicategories}\label{sec:rigid_Cstar_bicats}\label{sec:Cstarbicats}

In this section, we review some basic facts about rigid C$^*$-bicategories and some results in \cite{lr} and we set our notation along the way. Throughout the paper, we deal with categories and bicategories whose objects (0-morphisms) and $n$-morphisms, $n=1,2$, are sets. 

For a category $\CC$, we denote its objects by $X,Y,Z,\ldots$, we use $\CC(X,Y)$ to denote the set of morphisms from $X$ to $Y$. We denote by $1_X$ the unit morphism in $\CC(X,X)$. 

For a bicategory $\bB$, see Chapter XII in \cite{ml} for the definition, we use $\CL, \CM, \CN, \ldots$ to denote the objects and write 1-morphisms as $X,Y,Z,\ldots$. For two objects $\CM$ and $\CN$, the category of 1-morphisms from $\CM$ to $\CN$ is denoted by $\bB(\CM, \CN)$. The 2-morphisms in $\bB(\CM, \CN)(X,Y)$, where $X,Y \in \bB(\CM, \CN)$, will be denoted by $f,g,r,s,t,\alpha,\beta,\gamma,\ldots$. The horizontal composition (or tensor multiplication) functor 
\begin{align}\label{eq:hor_comp}
  \bB(\CM, \CN) \times \bB(\CL, \CM) \to \bB(\CL, \CN)
\end{align}
is written as $\otimes$.
More explicitly, for $X \in \bB(\CL, \CM)$ and $Y \in \bB(\CM, \CN)$ the product is $Y \otimes X \in \bB(\CL, \CN)$, as if $X:\CL \to \CM$, $Y:\CM \to \CN$ were functions and $Y \otimes X:\CL\to\CN$ their composition. 
For $s \in \bB(\CL, \CM)(X_1, X_2)$ and $t \in \bB(\CM, \CN)(Y_1, Y_2)$ we write $t \otimes s \in \bB(\CL, \CN)(Y_1 \otimes X_1,Y_2 \otimes X_2)$. 
For $\CM$ in $\bB$, we use $\one_\CM$ to denote the horizontal unit 1-morphism in $\bB(\CM, \CM)$. For $X\in \bB(\CL,\CM)$, we denote again by $1_X$ the vertical unit 2-morphism in $\bB(\CL,\CM)(X,X)$. The vertical composition of $f \in \bB(\CL, \CM)(X_1, X_2)$ and $g \in \bB(\CL, \CM)(X_2, X_3)$ is denoted by $gf \in \bB(\CL, \CM)(X_1, X_3)$.   

Recall that a C$^*$-category, see \cite{lro}, is a $\C$-linear category $\CC$ such that the sets of morphisms are complex Banach spaces with an antilinear contravariant $*$-map satisfying $f^{**} = f$, $(gf)^* = f^*g^*$ and $\|gf\| \leq \|f\| \|g\|$, $\|f^*f\| = \|f\|^2$, $f^*f \geq 0$ for every $f \in \CC(X,Y), g \in \CC(Y,Z)$. A $\C$-linear functor $F$ between C$^*$-categories is a $*$-functor if $F(f^*) = F(f)^*$. In particular, each $\CC(X,X)$ is a unital C$^*$-algebra. A C$^*$-category can always be completed to a C$^*$-category which is closed under finite direct sums and subobjects, see Appendix in \cite{lro}.

\begin{defn}
A bicategory $\bB$ is a \textbf{C$^*$-bicategory} if 
\begin{enumerate} 
\item $\bB(\CM, \CN)$ is a C$^*$-category which is closed under finite direct sums of 1-morphisms and sub-1-morphisms, for every pair of objects $\CM$, $\CN$ in $\bB$. The horizontal and vertical composition of 2-morphisms $(t, s) \mapsto t \otimes s$ and $(f, g) \mapsto gf$, whenever defined, are bilinear. The antilinear involutive contravariant $*$-map $s \in \bB(\CM, \CN)(X, Y) \mapsto s^* \in \bB(\CM, \CN)(Y, X)$ is compatible with the horizontal composition, i.e., $(t \otimes s)^* = t^* \otimes s^*$. 
\item The associators $\alpha_{Z,Y,X}:(Z \otimes Y) \otimes X \to Z \otimes (Y \otimes X)$ and the left and right unitors $\rho_X: \one_{\CN} \otimes X \to X$ and $\lambda_X: X \otimes \one_{\CM} \to X$, for $X\in \bB(\CM,\CN)$, are unitary. 
\end{enumerate}
\end{defn}

If all the associators and unitors are identity natural transformations, $\bB$ is called a \textit{strict} C$^*$-bicategory or \textit{2-C$^*$-category}, or also \textit{C$^*$ 2-category}. The definition of $2$-C$^*$-category appears in Section 7 of \cite{lro}, see also \cite{glr, z, bcls, lr}. By the coherence theorem for bicategories, see, e.g., Chapter 2 in \cite{ng}, every two parenthesized products of 1-morphisms with insertions of unit 1-morphisms can be canonically identified. Therefore we can safely ignore the associators and unitors in the calculations, or equivalently, we may assume that the bicategory is strict. We shall do so from now on, without further mention.

A bifunctor is a morphism between bicategories, see Section XII.6 in \cite{ml}, also called pseudofunctor in \cite{ben} and in Section 4.1 of \cite{jy}.

\begin{defn}
If $\bB$ and $\bB'$ are two C$^*$-bicategories, a bifunctor $F:\bB \to \bB'$ is a \textbf{$*$-bifunctor} if
\begin{enumerate}
\item Each local functor $F_{\CM,\CN} : \bB(\CM,\CN) \to \bB'(F(\CM),F(\CN))$ is a $*$-functor.
\item The laxity constraints, also called tensorator $F^2_{X,Y} : F_{\CM,\CN}(X) \otimes F_{\CL,\CM}(Y) \to F_{\CL,\CN}(X\otimes Y)$ and unitor $F^0_{\CM} : {\one}_{F(\CM)} \to F_{\CM,\CM}(\one_\CM)$ are unitary.
\end{enumerate}
A $*$-bifunctor is called a \textbf{$*$-biequivalence} if the underlying bifunctor is a biequivalence.
\end{defn}

Recall that in a C$^*$-bicategory, for every $X\in\bB(\CM, \CN)$ we have that $\bB(\CM, \CN)(X, X)$ is a unital C$^*$-algebra and $\bB(\CM, \CM)(\one_\CM, \one_\CM)$ is in addition abelian.

\begin{defn}\label{def:fin_dim_centers}
A C$^*$-bicategory $\bB$ is said to have \textbf{finite-dimensional centers} if every $\bB(\CM, \CM)(\one_\CM, \one_\CM)$ is finite-dimensional, see Definition 8.15 in \cite{lr}.
\end{defn}

\begin{defn}\label{def:rigid}
A C$^*$-bicategory $\bB$ is \textbf{rigid} if for every 1-morphism $X \in \bB(\CM, \CN)$, there exists a (unique up to isomorphism) 1-morphism $\overline{X} \in \bB(\CN, \CM)$, called a \textbf{dual} (or conjugate) of $X$, and a solution $(\omega,\overline{\omega})$ (a pair of 2-morphisms) of the so-called \textbf{conjugate equations} \eqref{equ:conjeq} for $X$ and $\overline{X}$. 

Namely, $\omega \in \bB(\CM,\CM)(\one_\CM, \overline{X} \otimes X)$ and $\overline{\omega} \in \bB(\CN,\CN)(\one_\CN, X \otimes \overline{X})$ are such that
\begin{align}\label{equ:conjeq}
  (X \xrightarrow{1_X \otimes \omega} X \otimes \overline{X} \otimes X \xrightarrow{\overline{\omega}^{\,*} \otimes 1_X} X) = 1_X, \quad
  (\overline{X} \xrightarrow{\omega \otimes 1_{\overline{X}}} \overline{X} \otimes X \otimes \overline{X} \xrightarrow{1_{\overline{X}} 
  \otimes \overline{\omega}^{\, *}} \overline{X}) = 1_{\overline{X}}.
\end{align}
\end{defn}

Recall that if $\bB$ is a rigid C$^*$-bicategory with finite-dimensional centers, then by Proposition 8.16 in \cite{lr}, $\bB(\CM, \CN)(X, Y)$ is a finite-dimensional vector space for every $X,Y \in \bB(\CM, \CN)$. Moreover, $\bB$ is semisimple, i.e., every 1-morphism in $\bB(\CM, \CN)$ is a finite direct sum of simple 1-morphisms in $\bB(\CM, \CN)$, and every $\bB(\CM, \CM)$ is a multitensor C$^*$-category. 

Recall that, by definition, a \textbf{multitensor C$^*$-category} is a \textit{rigid} C$^*$-bicategory with finite-dimensional centers which has only one object $\CM$. In particular, the tensor unit $\one := \one_\CM$ is finitely decomposable. A \textbf{tensor C$^*$-category} is a multitensor C$^*$-category with simple tensor unit.

More generally, a multitensor C$^*$-category is called \textbf{indecomposable} if it is neither zero nor the direct sum of two non-zero multitensor C$^*$-categories. Every multitensor C$^*$-category is a finite direct sum of indecomposable multitensor C$^*$-categories. See Chapter 4 in \cite{egno}.

Let $\CC$ be an \textit{indecomposable} multitensor C$^*$-category. Let $\one = \oplus_i\one_i$ where $\one_i$ are simple. Then by Theorem 2.5.1 in \cite{kz}, see also Remark 4.3.4 in \cite{egno}, $\CC$ has a decomposition $\CC = \oplus_{i,j} \CC_{ij}$ with $\CC_{ij} \dfeq \one_i \otimes \CC \otimes \one_j$ such that 
\begin{enumerate}
\item Every $\CC_{ij}$ is not zero.
\item Every simple object of $\CC$ belongs to some $\CC_{ij}$ and every $\CC_{ii}$ is a tensor C$^*$-category with tensor unit $\one_i$.
\item The tensor multiplication maps $\CC_{ij} \times \CC_{kl}$ into $\CC_{il}$, and the image is zero unless $j=k$. 
\item If $X \in \CC_{ij}$, then $\overline{X} \in \CC_{ji}$.
\end{enumerate}

\subsection{Standard solutions of the conjugate equations}\label{sec:stdsol}\,\\ \vspace{-2mm}

Let $X$ be a 1-morphism with dual $\overline{X}$ in a rigid C$^*$-bicategory $\bB$. Assume that $\bB$ has finite-dimensional centers. Among all the solutions $(\omega,\overline{\omega})$ of the conjugate equations \mref{equ:conjeq}, there is a particular class, called \textit{standard} solutions, with good functorial and optimization properties. They have been introduced in \cite{lro} for tensor C$^*$-categories, and extended in \cite{lr} to the case of multitensor categories and rigid 2-C$^*$-categories with finite-dimensional centers.
We recall the notion of \textit{connectedness} for 1-morphisms, see Definition 8.6 in \cite{lr}.
 
\begin{defn}
A 1-morphism $Y \in \bB(\CM, \CN)$ is called \textbf{connected} if 
\begin{align*}
    [1_Y \otimes \bB(\CM, \CM)(\one_\CM, \one_\CM)] \cap [\bB(\CN, \CN)(\one_\CN, \one_\CN) \otimes 1_Y] = \C 1_Y.
\end{align*}
\end{defn}

\begin{defn}\label{def:stdsolconn}
For a connected 1-morphism $Y \in \bB(\CM, \CN)$ with dual $\overline{Y} \in \bB(\CN,\CM)$, a solution $(\gamma, \overline{\gamma})$ of the conjugate equations \mref{equ:conjeq} is called \textbf{standard} (an appropriate terminology would also be minimal) if 
\begin{align*}
    \|\gamma\|^2 = \|\overline{\gamma}\|^2 = \inf_{(\omega, \overline{\omega})}\|\omega\|\|\overline{\omega}\|
\end{align*}
where the infimum is taken over all the solutions $(\omega, \overline{\omega})$ of the conjugate equations.
\end{defn}

\begin{rem}
The definition of standardness above is actually equivalent to the original Definition 8.29 in \cite{lr}, by Proposition 8.30 and Theorem 8.44 in \cite{lr}. We only need to observe that tensoring with $1_Y$ is isometric on $\omega^* \omega$ and $\overline{\omega}^{\,*} \overline{\omega}$, i.e., $\|1_Y \otimes (\omega^* \omega)\| = \|\omega^* \omega\|$ and $\|(\overline{\omega}^{\,*} \overline{\omega}) \otimes 1_Y\| = \|\overline{\omega}^{\,*} \overline{\omega}\|$, see, e.g., Lemma 1.13 in \cite{z}.

In the case at hand, standard solutions are shown to exist in \cite{lr} (realizing the minimum) and to be unique up to unitary isomorphism, see Lemma 8.35 in \cite{lr}. The notion of standard solutions relies on the more fundamental theory of \textbf{matrix dimension} for rigid 2-C$^*$-categories with finite-dimensional centers, see Definition 8.25 in \cite{lr}. The finite-dimensionality assumption guarantees the existence and uniqueness of the normalized Perron-Frobenius eigenvectors (in $L^2$) associated with the matrix dimension, which are crucial for the previously mentioned Definition 8.29 in \cite{lr}, see also Section 2 in \cite{g}. In more general cases than this, the existence of a good (functorial/optimal) notion of standardness for solutions of the conjugate equations in a rigid 2-C$^*$/W$^*$-category is yet to be established.
\end{rem}

By Lemma 8.19 in \cite{lr}, every 1-morphism $X \in \bB(\CM, \CN)$ is a finite direct sum of \textit{connected} 1-morphisms, i.e., there are finitely many connected $X_i \in \bB(\CM, \CN)$, uniquely determined up to unitary isomorphism and up to permutation of the indices, such that $X = \oplus_i X_i$. Moreover, $\overline{X}_i \in \bB(\CN,\CM)$ are connected, $\overline{X} = \oplus_i \overline{X}_i$ holds, and 
\begin{align*}
    X_i \otimes \overline{X}_j = \zero_\CN, \quad \overline{X}_j \otimes X_i = \zero_\CM, \quad i \neq j 
\end{align*}
where $\zero_\CM \in \bB(\CM,\CM)$ and $\zero_\CN \in \bB(\CN,\CN)$ are the \textit{zero} 1-morphisms. Thus 
\begin{align*}
\overline{X} \otimes X = \oplus_i (\overline{X}_i \otimes X_i), \quad X \otimes \overline{X} = \oplus_i (X_i \otimes \overline{X}_i). 
\end{align*}

\begin{defn}\label{def:stdsol}
For $X \in \bB(\CM, \CN)$ with dual $\overline{X} \in \bB(\CN,\CM)$, we define a \textbf{standard solution} $(\gamma_{X}, \overline{\gamma}_{X})$ of the conjugate equations \mref{equ:conjeq} as
\begin{align*}
    \gamma_X \dfeq \oplus_i \gamma_{X_i}, \quad \overline{\gamma}_X \dfeq \oplus_i \overline{\gamma}_{X_i}
\end{align*}
where $X_i$ and $\overline{X}_i$ are as above, and $(\gamma_{X_i}, \overline{\gamma}_{X_i})$ is a standard solution of the conjugate equations in the sense of Definition \ref{def:stdsolconn}.
The number $d_X \dfeq  \|\gamma_X\|^2$ is called the \textbf{scalar dimension} of $X$. It fulfills $d_X = d_{\overline{X}}$ and $d_X = \underset{i}{\max}\, d_{X_i}$. 
\end{defn}

\begin{rem}
Equivalently, the scalar dimension can be defined as the $L^2$-operator norm of the matrix dimension, see Definition 8.25 and Proposition 8.30 in \cite{lr}.
\end{rem}

The following facts about tensoring standard solutions are not explicitly stated in \cite{lr}. We prove them for later use. 

\begin{lem}\label{lem:dsubmult}
For a connected 1-morphism $X \in \bB(\CM, \CN)$ with dual $\overline{X} \in \bB(\CN,\CM)$, we have $d_{\overline{X} \otimes X} = d_X^2$. 
More generally, for connected 1-morphisms $X \in \bB(\CM, \CN)$ and $Y \in \bB(\CL, \CM)$, we have $d_{X\otimes Y} \leq d_{X}d_{Y}$.
\end{lem}

\begin{proof}
By Proposition 8.47 and Remark 8.26 in \cite{lr}, the matrix dimension function is multiplicative and it turns dualization to transposition. Thus the first equality follows from the C$^*$-property of the $L^2$-operator norm, or from the analogue of Proposition 6.6 in \cite{lr} whose assumptions are fulfilled in the case of $X$ and $\overline{X}$. 
The inequality in the second part of the statement follows from the submultiplicativity of the $L^2$-operator norm.
\end{proof}

\begin{prop}\label{prop:CX_X_solution}
For every $X \in \bB(\CM, \CN)$, $\overline{X} \otimes X$ is self-dual and 
\begin{align*}
  \gamma = \overline{\gamma} \dfeq (1_{\overline{X}} \otimes \overline{\gamma}_X \otimes 1_{X}) \gamma_X 
\end{align*}
is a standard solution of the conjugate equations \mref{equ:conjeq}, where $(\gamma_X,\overline{\gamma}_X)$ is a standard solution of the conjugate equations for $X$ and $\overline{X}$.
\end{prop}

\begin{proof}
It is easy to check that $\overline{X} \otimes X \otimes \overline{X} \otimes X = \oplus_i (\overline{X}_i \otimes X_i \otimes \overline{X}_i \otimes X_i)$ and 
\begin{align*}
  \gamma = \oplus_i (1_{\overline{X}_i} \otimes \overline{\gamma}_{X_i} \otimes 1_{X_i})\gamma_{X_i}
\end{align*}
where $X_i$ and $\overline{X}_i$ are the connected components of $X$ and $\overline{X}$, respectively.
By Remark 8.26 and Lemma 8.27 in \cite{lr}, $\overline{X}_i \otimes X_i$ is connected.
Thus the scalar dimension of $\overline{X}_i \otimes X_i$ is the square of the scalar dimension of $X_i$ by the first statement of Lemma \ref{lem:dsubmult}. By Proposition 8.30 in \cite{lr}, $(1_{\overline{X}_i} \otimes \overline{\gamma}_{X_i} \otimes 1_{X_i})\gamma_{X_i}$ gives a standard solution of the conjugate equations for $\overline{X}_i \otimes X_i$, for every $i$, in the sense of Definition \ref{def:stdsolconn}. Thus, $(\gamma,\gamma)$ is a standard solution of the conjugate equations in the sense of Definition \ref{def:stdsol}.
\end{proof}

\begin{rem}
Note that the inequality $d_{X\otimes Y} \leq d_{X}d_{Y}$ in Lemma \ref{lem:dsubmult} is in general strict, see Corollary 6.5 in \cite{lr}. It is an equality when the sufficient conditions of Proposition 6.6 in \cite{lr} are fulfilled, as it is for example the case for $X$ and $Y=\overline{X}$. Moreover, $(1_{\overline{Y}} \otimes \gamma_X \otimes 1_{Y}) \gamma_Y$ and $(1_X \otimes \overline{\gamma}_X \otimes 1_{X}) \gamma_X$ is a solution of the conjugate equations for $X \otimes Y$ and $\overline{Y} \otimes \overline{X}$, but in general \textit{not} a standard one.
\end{rem}

\subsection{Special C$^*$-Frobenius algebras and bimodules}\label{sec:specialalgs}\,\\ \vspace{-2mm}

In this section, we recall the definition of algebra in a C$^*$-bicategory $\bB$, algebra homomorphism, bimodule and bimodule map, and generalize some results from \cite{bklr} to our setting.

An \textbf{algebra} in $\bB$, or more precisely in $\bB(\CM, \CM)$ for an object $\CM$, is a triple $(A, m, \iota)$ consisting of a 1-morphism $A \in \bB(\CM, \CM)$, and 2-morphisms $m \in \bB(\CM,\CM)(A \otimes A, A)$ (the multiplication of $A$) and $\iota \in \bB(\CM,\CM)(\one_\CM ,A)$ (the unit of $A$) such that 
\begin{align*}
  m(\iota \otimes 1_A) = 1_A = m(1_A \otimes \iota), \quad m(m \otimes 1_A) = m(1_A \otimes m). 
\end{align*}
Let $(A', m', \iota')$ be another algebra in $\bB(\CM, \CM)$. A 2-morphism $t \in \bB(\CM,\CM)(A,A')$ is called an \textbf{algebra homomorphism} from $A$ to $A'$ if $tm = m'(t \otimes t)$ and $t \iota = \iota'$.

Let $(A, m_A, \iota_A)$ and $(B, m_B, \iota_B)$ be two algebras in $\bB(\CM, \CM)$ and $\bB(\CN, \CN)$, respectively. A \textbf{$B$-$A$-bimodule} is a triple $(X, l_X, r_X)$, also denoted by $(X, l, r)$ or just $X$ if no confusion arises, consisting of a 1-morphism $X\in\bB(\CM, \CN)$, and 2-morphisms $l \in \bB(\CM, \CN)(B \otimes X, X)$ (the left action of $B$ on $X$) and $r \in \bB(\CM, \CN)(X \otimes A, X)$ (the right action of $A$ on $X$) such that
\begin{align*}
  &l(m_B \otimes 1_X) = l(1_B \otimes l), \quad r(1_X \otimes m_A) = r(r \otimes 1_A),\\    
  &l(\iota_B \otimes 1_X) = 1_X = r(1_X \otimes \iota_A), \quad r(l \otimes 1_A) = l(1_B \otimes r).
\end{align*}
In particular, an algebra $A$ bears the structure of an $A$-$A$-bimodule with the left and right actions given by the multiplication. 

\begin{defn}\label{def:specialX}
We say that a $B$-$A$-bimodule $X$ is \textbf{special} if 
\begin{align*}
ll^* = 1_X = rr^*.
\end{align*}
\end{defn}

Let $(X_i, l_i, r_i)$ be two $B$-$A$-bimodules, $i=1,2$. A 2-morphism $f \in \bB(\CM,\CN)(X_1, X_2)$ is called a \textbf{left $B$-module map} from $X_1$ to $X_2$ if $fl_1 = l_2(1_B \otimes f)$. Analogously, $g \in \bB(\CM,\CN)(X_1, X_2)$ is called a \textbf{right $A$-module map} from $X_1$ to $X_2$ if $gr_1 = r_2(g \otimes 1_A)$. Finally, a \textbf{$B$-$A$-bimodule map} is a left $B$-module map and a right $A$-module map.

\begin{notat}\label{not:BAisAtoB}
In order to be consistent with the notation for the relative tensor product of bimodules over von Neumann algebras in Section \ref{sec:connes_bim}, and with the one in Section \ref{sec:rigid_Cstar_bicats}, we think of a $B$-$A$-bimodule as a 1-morphism from the algebra $A$ to the algebra $B$. We denote by $\BMod_\bB(A,B)$ the category of $B$-$A$-bimodules and $B$-$A$-bimodule maps, over the algebras $A$, $B$ in $\bB$. 
\end{notat}

\begin{defn}
An algebra $(A, m, \iota)$ is called a \textbf{C$^*$-Frobenius algebra} if it fulfills the Frobenius property, namely if 
\begin{align*}
  (1_A \otimes m)(m^* \otimes 1_A) = m^*m =  (m \otimes 1_A)(1_A \otimes m^*)
\end{align*}
i.e., if $m^*$ is an $A$-$A$-bimodule map.

An algebra $(A, m, \iota)$ is called \textbf{special} if $A$ is a special $A$-$A$-bimodule, i.e., if 
\begin{align*}
mm^* = 1_A.
\end{align*}
\end{defn}  

By Lemma 3.7 in \cite{bklr}, a special algebra is automatically a C$^*$-Frobenius algebra. The Frobenius and unit properties imply that a C$^*$-Frobenius algebra $A$ is a self-dual 1-morphism in $\bB$ in the sense of Definition \ref{def:rigid}, with a solution of the conjugate equations \eqref{equ:conjeq} given by $(m^* \iota, m^* \iota)$. 

For the rest of this section, we collect some basic facts about C$^*$-Frobenius algebras in $\bB$ and their bimodules. The following proposition generalizes Lemma 3.5 and Corollary 3.6 in \cite{bklr}.

\begin{prop}\label{lem:lr_module_map_induce_iso}
Let $(A, m, \iota)$ be a C$^*$-Frobenius algebra in $\bB$. If $t$ is an invertible left (or right) $A$-module map from $A$ to $A$, then $(A, tm(t^{-1} \otimes t^{-1}), t \iota)$ is a C$^*$-Frobenius algebra isomorphic to $(A, m, \iota)$.
 
In particular, every C$^*$-Frobenius algebra $(A, m, \iota)$ in $\bB$ is isomorphic to a special one $(A, n m(n^{-1} \otimes n^{-1}), n \iota)$, where $n \dfeq (mm^*)^{1/2}$ is an invertible positive $A$-$A$-bimodule map. 
\end{prop}

\begin{proof}
Although this fact is well known to experts, we sketch the proof for the sake of completeness. Assume that $t$ is a left $A$-module map, i.e., $ tm = m (1_A \otimes t)$, then the Frobenius and unit properties imply that
\begin{align*}
  t^* m = m (1_A \otimes t^*).
\end{align*}
Similarly if $t$ is a right $A$-module map. Now it is easy to check that $(A, tm(t^{-1} \otimes t^{-1}), t \iota)$ is a C$^*$-Frobenius algebra and that $t$ is an algebra isomorphism. Thus the first statement is proved. Since $1_A \geq \frac{1}{\|\iota\|^2} \iota\iota^*$ and $A$ is an algebra, we have
\begin{align*}
  mm^* \geq \frac{1}{\|\iota\|^2}m((\iota\iota^*) \otimes 1_A)m^* = \frac{1}{\|\iota\|^2} 1_A. 
\end{align*}
Then, arguing as in Lemma 3.5 in \cite{bklr}, the Frobenius and associativity properties imply that $mm^*$ is an invertible $A$-$A$-bimodule map. Specialness can be checked with the same calculation leading to Corollary 3.6 in \cite{bklr}. Thus the second statement is a consequence of the first.\end{proof}

Let $(A, m_A, \iota_A)$ and $(B, m_B, \iota_B)$ be two \textit{special} C$^*$-Frobenius algebras in $\bB$, and let $(X, l, r)$ be a $B$-$A$-bimodule. By an argument similar to the one leading to Proposition \ref{lem:lr_module_map_induce_iso} we have that $lrr^*l^*= rll^*r^*$ is invertible and positive, and 
\begin{align*}
  (X, h^{-1}l(1_B \otimes h), h^{-1}r(h \otimes 1_A))  
\end{align*}
is a special $B$-$A$-bimodule, where $h \dfeq (lrr^*l^*)^{1/2}$. Moreover, $h^{-1}$ is a $B$-$A$-bimodule isomorphism from $(X, l,r)$ to $(X, h^{-1}l(1_B \otimes h), h^{-1}r(h \otimes 1_A))$. Thus every bimodule over special C$^*$-Frobenius algebras is isomorphic to a special one, cf.\ Lemma 3.22 in \cite{bklr}. 

\begin{notat}
In the following, we use $\sBMod_\bB(A,B)$ to denote the full subcategory of $\BMod_\bB(A,B)$ consisting of all special $B$-$A$-bimodules and $B$-$A$-bimodule maps, for fixed algebras $A$, $B$ in $\bB$. 
\end{notat}

From the discussion above, if $A$ and $B$ are special C$^*$-Frobenius algebras, we know that $\BMod_\bB(A,B)$ is equivalent to $\sBMod_\bB(A,B)$ as a $\C$-linear category. If $(X,l,r)$ is a special $B$-$A$-bimodule, then by exactly the same line of arguments as in the proof of Lemma 3.23 in \cite{bklr} it is not hard to check that 
\begin{align}
  &(m_B \otimes 1_X)(1_B \otimes l^*) = l^* l = (1_B \otimes l)(m_B^* \otimes 1_X), \label{equ:l_module_f}\\
  &(1_X \otimes m_A)(r^* \otimes 1_A) = r^* r = (r \otimes 1_A) (1_X \otimes m_A^*). \label{equ:r_module_f}
\end{align}
By equations (\ref{equ:l_module_f}) and (\ref{equ:r_module_f}), we have that $f$ is a $B$-$A$-bimodule map between two special $B$-$A$-bimodules $X_1$ and $X_2$ if and only if $f^*$ is a $B$-$A$-bimodule map between $X_2$ and $X_1$. In particular, $\sBMod_\bB(A,B)$ is a C$^*$-category.

\subsection{The C$^*$-bicategory $\sBMod_\bB$}\label{sec:specialbimods}\,\\ \vspace{-2mm}

Let $\bB$ be a C$^*$-bicategory as above. From now on we assume in addition that it is rigid in the sense of Definition \ref{def:rigid}. 
Proceeding as in Section 4 and 5 in \cite{y}, it can be shown that \textit{special} C$^*$-Frobenius algebras in $\bB$ (as objects), their bimodules (as 1-morphisms) and bimodule maps (as 2-morphisms) constitute a \emph{rigid} bicategory, denoted by $\BMod_\bB$. 
In general, the $*$-structure on $\bB$ does \emph{not} define a $*$-structure on $\BMod_\bB$ since $\BMod_\bB(A,B)$ need not be closed under the $*$-map on bimodule maps $f \mapsto f^*$.

To continue our discussion, we need to describe the sub-bicategory $\sBMod_\bB$ whose 1-morphisms are \textit{special} bimodules, see Definition \ref{def:specialX}, over special C$^*$-Frobenius algebras in more details. We point out that most of the results of this section are well known to experts, at least when $\bB$ is a tensor C$^*$-category.

Let $(A, m_A, \iota_A) \in \bB(\CL, \CL)$, $(B, m_B, \iota_B) \in \bB(\CM, \CM)$ and $(C, m_C, \iota_C) \in \bB(\CN, \CN)$ be special C$^*$-Frobenius algebras. Let $(X, l_X, r_X) \in \sBMod_\bB(A,B)$ and $(Y, l_Y, r_Y) \in \sBMod_\bB(B,C)$. By equation (\ref{equ:l_module_f}) and (\ref{equ:r_module_f})
\begin{align}\label{equ:coequ_proj}
  p^B_{Y \otimes X} \dfeq (r_Y \otimes 1_X)(1_Y \otimes l_X^*) = (r_Y \otimes l_X)[1_Y \otimes (m_B^* \iota_B) \otimes 1_X] = (1_Y \otimes l_X)(r_Y^* \otimes 1_X) 
\end{align}
is a projection in $\sBMod_\bB(A,C)(Y \otimes X, Y\otimes X)$, cf.\ Lemma 3.36 in \cite{bklr}. It fulfills  
\begin{align}\label{equ:2-cells_fusion}
  (g \otimes f)(p^B_{Y_1 \otimes X_1}) = (p^B_{Y_2 \otimes X_2})(g \otimes f)
\end{align}
for every $f \in \sBMod_\bB(A,B)(X_1, X_2)$ and $g \in \sBMod_\bB(B,C)(Y_1, Y_2)$. 

\begin{notat}
In the following, we denote by $Y \rt{B} X$ the sub-1-morphism in $\bB(\CL,\CN)$ of $Y \otimes X$ associated with the projection $p^B_{Y \otimes X}$ and a chosen isometry $s\in\bB(\CL,\CN)(Y \rt{B} X, Y \otimes X)$ such that $s{s}^* = p^B_{Y \otimes X}$. With a routine calculation, one can check that 
    \begin{align*}
        \left (Y \rt{B} X, s^*(l_{Y} \otimes 1_{X})(1_C \otimes s), s^*(1_{Y} \otimes r_{X})(s \otimes 1_A) \right )  
    \end{align*}
is a special $C$-$A$-bimodule. We denote by $g \rt{B} f$ the 2-morphism in $\bB(\CL,\CN)(Y_1 \rt{B} X_1, Y_2 \rt{B} X_2)$ defined by compressing the two members of equation \mref{equ:2-cells_fusion} with the defining isometries $s_1\in\bB(\CL,\CN)(Y_1 \rt{B} X_1, Y_1 \otimes X_1)$ and $s_2\in\bB(\CL,\CN)(Y_2 \rt{B} X_2, Y_2 \otimes X_2)$ such that $s_1{s_1}^* = p^B_{Y_1 \otimes X_1}$ and $s_2{s_2}^* = p^B_{Y_2 \otimes X_2}$. It is clear that $(g_2 \rt{B} f_2)(g_1 \rt{B} f_1) = (g_2g_1) \rt{B} (f_2 f_1)$ and $(g \rt{B} f)^* = g^* \rt{B} f^*$. Moreover, $g \rt{B} f \in \sBMod_\bB(A,C)(Y_1 \rt{B} X_1, Y_2 \rt{B} X_2)$. 
\end{notat}

By equations \mref{equ:l_module_f}, \mref{equ:r_module_f} and by the commutation of left and right actions in a bimodule, we have
\begin{align*}
p^C_{C \otimes Y} = l_Y^*l_Y, \quad p^A_{X \otimes A} = r_X^* r_X, \quad (p^C_{Z \otimes Y} \otimes 1_X)(1_{Z} \otimes p^B_{Y \otimes X}) =  (1_{Z} \otimes p^B_{Y \otimes X})(p^C_{Z \otimes Y} \otimes 1_X). 
\end{align*}
Thus $C \rt{C} Y$, $X \rt{A} A$, with chosen isometries $l_Y^*$, $r_X^*$, coincide with $Y$, $X$, respectively, and $(Z\rt{C}Y) \rt{B} X$ is canonically unitarily isomorphic to $Z \rt{C} (Y \rt{B} X)$. Cf.\ \cite{y}.

Summing up the discussion, we have shown the following

\begin{prop}\label{prop:special_Fro_alg_bicat}
Special C$^*$-Frobenius algebras in $\bB$, their special bimodules and bimodule maps constitute a C$^*$-bicategory denoted by $\sBMod_\bB$. The horizontal unit 1-morphism in $\sBMod_\bB(A,A)$ is $A$, with left and right actions induced by the multiplication. The horizontal composition $\sBMod_\bB(B,C)\times \sBMod_\bB(A,B) \to \sBMod_\bB(A,C)$ is given by the $*$-bifunctor $- \rt{B} -$.
\end{prop}  

\begin{rem}\label{rem:BinsBiModB}
The C$^*$-bicategory $\bB$ can be viewed as a full sub-C$^*$-bicategory of $\sBMod_\bB$. 
Indeed, $\one_\CM$ is a special C$^*$-Frobenius algebra in $\bB(\CM,\CM)$ for every object $\CM$, and every $X \in \bB(\CM, \CN)$ is a special $\one_\CN$-$\one_\CM$-bimodule. It is clear that $p^{\one_\CM}_{Y \otimes X} = 1_{Y \otimes X}$. Thus $Y \rt{\one_\CM} X = Y \otimes X$ and $\sBMod_\bB(\one_\CM,\one_\CN) =\bB(\CM, \CN)$.
\end{rem}

Recall that $\sBMod_\bB(A,B)$ is equivalent to $\BMod_\bB(A,B)$ for $A$, $B$ special. Thus $\sBMod_\bB$ is biequivalent to $\BMod_\bB$. Since $\BMod_\bB$ is rigid, $\sBMod_\bB$ is also rigid. Moreover, $\sBMod_\bB$ has finite-dimensional centers whenever $\bB$ has, since all the 2-morphisms spaces $\bB(\CM,\CN)(X,Y)$ are finite-dimensional as observed in Section \ref{sec:rigid_Cstar_bicats}.

\subsection{Multitensor C$^*$-categories as bimodule categories}\label{sec:multisbimcat}\,\\ \vspace{-2mm}

The goal of this section is to show that every indecomposable multitensor C$^*$-category, see Section \ref{sec:rigid_Cstar_bicats}, is equivalent to the special bimodule category over a fixed \textit{standard} C$^*$-Frobenius algebra in a tensor C$^*$-category. In other words, multitensor C$^*$-categories are finite-dimensional amplifications of tensor C$^*$-categories.

We first introduce the notion of standard C$^*$-Frobenius algebra. Let $\bB$ be a rigid C$^*$-bicategory and assume in addition that it has finite-dimensional centers in the sense of Definition \ref{def:fin_dim_centers}.

\begin{defn}\label{def:standardA}
A C$^*$-Frobenius algebra $(A,m,\iota)$
is called \textbf{standard} if it is special and if $(m^* \iota, m^* \iota)$ is a standard solution of the conjugate equations \eqref{equ:conjeq} for $A$, in the sense of Definition \ref{def:stdsol}.
\end{defn}

\begin{prop}
Let $(A, m, \iota)$ be a special C$^*$-Frobenius algebras in $\bB(\CM, \CM)$, and denote by $X$ the special $A$-$\one_\CM$-bimodule $(A, m, 1_A)$, and by $\overline{X}$ the special $\one_\CM$-$A$-bimodule $(A, 1_A, m)$. Then there is an invertible $\one_\CM$-$A$-bimodule map $t$ from $\overline{X}$ to $\overline{X}$ such that $((t \otimes (t^{-1})^*)m^* \iota, (t \otimes (t^{-1})^*)m^* \iota)$ is a standard solution of the conjugate equations for $A$.

In particular, every special C$^*$-Frobenius algebra in $\bB$ is isomorphic to a standard one. 
\end{prop}

\begin{proof}
Since $m^*$ is an isometry with range projection $p^A_{\overline{X} \otimes X} = m^*m$, we have that $\overline{X} \rt{A} X = A$. On the other hand, $X \rt{\one_{\CM}} \overline{X} = A \otimes A$. It is clear that $\iota \in \bB(\CM,\CM)(\one_\CM, A)$ is a $\one_\CM$-$\one_\CM$-bimodule map, and that $m^*\in\bB(\CM,\CM)(A, A \otimes A)$ is an $A$-$A$-bimodule map. Moreover, $\one_{\one_\CM} = \one_\CM$ and $\one_{A} = A$ are the tensor unit bimodules. Thus $(\iota, m^*)$ is a solution of the conjugate equations for $X$ and $\overline{X}$ in $\sBMod_\bB$. Recall that $\sBMod_\bB$ is a rigid C$^*$-bicategory with finite-dimensional centers. By Lemma 8.12 in \cite{lr}, there is an invertible $\one_\CM$-$A$-bimodule map $t$ from $\overline{X}$ to $\overline{X}$ such that $((t \rt{A} 1_{X})\iota, (1_X \rt{\one_\CM} (t^{-1})^*)m^*)$ is a standard solution of the conjugate equations for $X$ and $\overline{X}$ in $\sBMod_\bB$. Note that $(t \rt{A} 1_X)\iota = t\iota$ and recall that $m^*t = (t \otimes 1_A)m^*$. By Proposition \ref{prop:CX_X_solution}, $((t \otimes (t^{-1})^*)m^* \iota, (t \otimes (t^{-1})^*)m^* \iota)$ is a standard solution of the conjugate equations for the self-dual bimodule $\overline{X} \rt{A} X = A$. By Proposition \ref{lem:lr_module_map_induce_iso}, $(A, m, \iota)$ is isomorphic via the right $A$-module map $t$ to the C$^*$-Frobenius algebra $(A, tm(t^{-1} \otimes t^{-1}), t\iota) = (A, m(1_A \otimes t^{-1}), t\iota)$. Invoking again Proposition \ref{lem:lr_module_map_induce_iso}, we know that $(A, m(1_A \otimes t^{-1}), t\iota)$ is isomorphic to a special C$^*$-Frobenius algebra via the $A$-$A$-bimodule map $n = (m(1_A \otimes (t^*t)^{-1})m^*)^{1/2}$. The latter C$^*$-Frobenius algebra is also standard by our choice of $t$ and because $n$ is $A$-$A$-bimodular.
\end{proof}

In the remainder of this section we shall mainly be concerned with multitensor C$^*$-categories, e.g., $\bB(\CM,\CM)$ where $\bB$ is as above, but we state the following two lemmas in more generality for later use. 

\begin{lem}\label{lem:indecom_f_alg}
Assume that $\bB(\CM, \CM)$ and $\bB(\CN, \CN)$ are indecomposable and $\bB(\CM, \CN) \neq 0$. Then there is a non-zero $X \in \bB(\CM, \CN)$ and a solution $(\gamma, \overline{\gamma})$ of the conjugate equations \mref{equ:conjeq} for $X$ and $\overline{X}$, such that $\overline{\gamma}^{\, *}\overline{\gamma} = 1_{\one_\CN}$. 
\end{lem}

\begin{proof}
Let $\one_{\CM} = \oplus_{i} e_i$ and $\one_{\CN} = \oplus_{j} f_j$, where $e_i$ and $f_j$ are the simple sub-1-morphisms of $\one_\CM$ and $\one_\CN$, respectively. Since $\bB(\CM, \CN) \neq 0$, we may choose $X_0 \in \bB(\CM, \CN)$ such that $X_0 = f_1 \otimes X_0 \otimes e_1 \neq 0$. Since $\bB(\CN, \CN)$ is indecomposable, 
we can also choose non-zero $Y_{j} \in f_j \otimes \bB(\CN, \CN) \otimes f_1$ for every $j$. By setting $X = \oplus_{j} (Y_j \otimes X_0)$ we get a connected 1-morphism in $\bB(\CM, \CN)$. Let $(\gamma_0, \overline{\gamma}_0)$ be a standard solution of the conjugate equations for $X$ and $\overline{X}$. Since $X$ is connected and $f_j \otimes X \neq 0$ for every $j$, by Proposition 8.30 in \cite{lr}, we have that $\overline{\gamma}_0^{\, *}\overline{\gamma}_0 = d_X 1_{\one_\CN}$. Then $(\gamma, \overline{\gamma}) = (d_X^{1/2}\gamma_0, d_X^{-1/2}\overline{\gamma}_0)$ is a solution of the conjugate equations with the desired normalization.
\end{proof}

Observe that 
\begin{align*}
(\overline{X} \otimes X, 1_{\overline{X}} \otimes \overline{\gamma}^{\, *} \otimes 1_X, \gamma) 
\end{align*}
is a special C$^*$-Frobenius algebra in $\bB(\CM, \CM)$, which is also standard, in the sense of Definition \ref{def:standardA}, by Proposition \ref{prop:CX_X_solution}.

\begin{lem}\label{lem:2-cat_std_Frob_bimod_equ}
Fix an object $\CM$ in $\bB$. Assume that for every other object $\CN$ in $\bB$ there is a non-zero $X_\CN \in \bB(\CM, \CN)$ and a solution $(\gamma_\CN, \overline{\gamma}_\CN)$ of the conjugate equations \mref{equ:conjeq} for $X_{\CN}$ and $\overline{X}_\CN$ such that $\overline{\gamma}^{\, *}_\CN\overline{\gamma}_\CN = 1_{\one_{\CN}}$. 

Then there is a locally\footnote{\textit{locally} means referred to each $*$-functor $\Theta_{\CN,\CL}$} essentially surjective and fully faithful $*$-bifunctor 
\begin{align*}
\Theta: \bB \to \sBMod_{\bB(\CM, \CM)}
\end{align*}
defined as follows
\begin{enumerate}
\item $\Theta(\CN)$ is the special C$^*$-Frobenius algebra in $\bB(\CM, \CM)$ defined by
\begin{align*}
  \Theta(\CN) \dfeq (\overline{X}_\CN\otimes X_\CN, 1_{\overline{X}_\CN} \otimes \overline{\gamma}^{\, *}_\CN \otimes 1_{X_\CN}, \gamma_\CN).
\end{align*}
\item $\Theta_{\CN, \CL}: \bB(\CN, \CL) \to \sBMod_{\bB(\CM, \CM)}(\Theta(\CN), \Theta(\CL))$
is the $*$-functor defined by 
\begin{align*}
  Y \mapsto \Theta_{\CN, \CL}(Y) &\dfeq \overline{X}_\CL \otimes Y \otimes X_\CN\\
  t \mapsto \Theta_{\CN, \CL}(t) &\dfeq 1_{\overline{X}_\CL} \otimes t \otimes 1_{X_\CN}
\end{align*} 
where $\Theta_{\CN, \CL}(Y)$ bears the structure of a $\Theta(\CL)$-$\Theta(\CN)$-bimodule with left and right actions respectively induced by $\overline{\gamma}^{\, *}_{\CL}$ and $\overline{\gamma}^{\, *}_{\CN}$.
\end{enumerate}
\end{lem}

\begin{proof}
Proceeding as in the proof of Proposition 3.24 and 3.25 in \cite{bklr}, one can show that $\Theta_{\CN, \CL}$ is an essentially surjective and fully faithful $*$-functor for every pair of objects $\CL$, $\CN$. We only have to observe that the maps $t \mapsto 1_{\overline{X}_\CL} \otimes t$ and $t \mapsto t \otimes 1_{X_\CN}$ are injective due to the conditions $\overline{\gamma}^{\, *}_\CL\overline{\gamma}_\CL = 1_{\one_{\CL}}$ and $\overline{\gamma}^{\, *}_\CN\overline{\gamma}_\CN = 1_{\one_{\CN}}$.

To show the $*$-bifunctoriality of $\Theta$, note that the projection defined in equation \mref{equ:coequ_proj} reads in this case
\begin{align*}
  p^{\Theta(\CL)}_{\Theta_{\CL, \CP}(Z) \,\otimes\, \Theta_{\CN, \CL}(Y)} = (1_{\overline{X}_{\CP} \otimes Z} \otimes \overline{\gamma}_{\CL} \otimes 1_{Y \otimes X_{\CN}})(1_{\overline{X}_{\CP} \otimes Z} \otimes \overline{\gamma}^{\, *}_{\CL} \otimes 1_{Y \otimes X_{\CN}}).   
\end{align*} 
Therefore, we have $\Theta_{\CL, \CP}(Z) \rt{\Theta(\CL)} \Theta_{\CN, \CL}(Y) = \overline{X}_{\CP} \otimes Z \otimes Y \otimes X_{\CN} = \Theta_{\CN, \CP}(Z \otimes Y)$ and $\Theta_{\CL, \CP}(s) \rt{\Theta(\CL)} \Theta_{\CN, \CL}(t) = \Theta_{\CN, \CP}(s \otimes t)$.

Moreover, $\Theta_{\CN, \CN}(\one_\CN) = \overline{X}_{\CN} \otimes X_\CN = \one_{\Theta(\CN)}$. Thus $\Theta$ is a $*$-bifunctor.
\end{proof}

Let $\CC$ be an indecomposable multitensor C$^*$-category and let $\CC = \oplus_{i,j}\CC_{ij}$ be the decomposition mentioned at the beginning of this section. We can apply Lemma \ref{lem:2-cat_std_Frob_bimod_equ} and \ref{lem:indecom_f_alg} to $\CC=\bB(\CM,\CM)$ viewed as a C$^*$-bicategory with only one object, thus $\CM=\CN=\CL$. By observing that the connected $X$ constructed in the proof of \ref{lem:indecom_f_alg} fulfills $\overline{X}\otimes X = e_1 \otimes \overline{X}\otimes X\otimes e_1$, thus $\overline{X}\otimes X\in\CC_{11}$, setting $H := \overline{X}\otimes X$ and $\CC_0 := \CC_{11}$ we conclude

\begin{cor}\label{cor:ind_ten_equ_alg_ten}
For every indecomposable multitensor C$^*$-category $\CC$, there exist a tensor C$^*$-category $\CC_0$ and a standard C$^*$-Frobenius algebra $(H, m, \iota)$ in $\CC_0$ such that $\CC \simeq \sBMod_{\CC_0}(H,H)$ as multitensor C$^*$-categories. 
\end{cor}

\section{Realization of multitensor C$^*$-categories}
\subsection{Connes' bimodules}\label{sec:connes_bim}\,\\ \vspace{-2mm}

In this section we review some basic facts about Connes' bimodules (or correspondences) over von Neumann algebras \cite{acncg}, specializing to direct sums of II$_1$ factors \cite{gdhj}. Let $\CM = \oplus_i \CM_i$ and $\CN = \oplus_j\CN_j$, where $\CM_i$ and $\CN_j$ are II$_1$ factors and $i,j$ run in two finite sets of indices. Denote by $p_i$ and $q_j$ the corresponding minimal central projections in $\CM$ and $\CN$, respectively. 

An \textbf{$\CN$-$\CM$-bimodule} is a Hilbert space $X$ with commuting left $\CN$-action and right $\CM$-action given by normal unital (not necessarily faithful) representations of $\CN$ and $\CM^{op}$ on $X$, where $\CM^{op}$ is the opposite von Neumann algebra of $\CM$. It is clear $X$ is an $\CN$-$\CM$-bimodule if and only if $X = \oplus_{ji} X_{ji}$ where $X_{ji} \dfeq q_j X p_i$ are $\CN_j$-$\CM_i$-bimodules. Note that $X_{ji}$ may be zero. The left $\CN$-action and right $\CM$-action are recovered by 
\begin{align*}
  (\oplus_j b_j) (\oplus_{ji}\xi_{ji})(\oplus_i a_i) \dfeq \oplus_{ji} (b_j \xi_{ji} a_i)
\end{align*}
where $a_i \in \CM_i$, $b_j \in \CN_j$ and $\xi_{ji} \in X_{ji}$. For two $\CN$-$\CM$-bimodules $X = \oplus_{ji} X_{ji}$ and $X' = \oplus_{ji} X'_{ji}$, a bounded operator $f: X \to X'$ is an $\CN$-$\CM$-bimodule \textbf{intertwiner} if and only if $f = \oplus_{ji} f_{ji}$, where $f_{ji}: X_{ji} \to X'_{ji}$ are $\CN_j$-$\CM_i$-bimodule intertwiners. An $\CN$-$\CM$-bimodule $X = \oplus_{ji} X_{ji}$ is \textbf{finite} if and only if $X_{ji}$ are finite $\CN_j$-$\CM_i$-bimodules. See Chapter 9 in \cite{ek} for the definition of finite bimodules over II$_1$ factors. 

\begin{notat}
The category of all finite $\CN$-$\CM$-bimodules with their intertwiners as morphisms is denoted by $\vN(\CM,\CN)$. For the choice of ordering cf.\ Notation \ref{not:BAisAtoB}.
\end{notat}

In the following, we use $\tau_\CM$ to denote the normal faithful trace on $\CM$ such that its restriction to $\CM_i$ is the unique tracial state for every $i$. For $a \in \CM$, denote $\|a\|_2 \dfeq \sqrt{\tau_\CM (a^*a)}$. Recall that the set of \textbf{$\CM$-bounded}, or equivalently \textbf{$\CN$-bounded}, \textbf{vectors} in a finite $\CN$-$\CM$-bimodule $X$ is 
\begin{align*}
  X^{\bdd} &\dfeq \{\xi \in X: \|\xi a\| \leq C_\xi \|a\|_2, C_\xi > 0, \forall a \in \CM\}\\
  &= \{\xi \in X: \|b\xi\| \leq C'_\xi \|b\|_2, C'_\xi > 0, \forall b \in \CN\}
\end{align*}
see Proposition 9.57 in \cite{ek}. 
If $X = \oplus_{ji}X_{ji}$, then it is clear that $X^{\bdd} =\oplus_{ji}X_{ji}^{\bdd}$. By the Sakai-Randon-Nikdo\'{y}m theorem, see Theorem 7.36 in \cite{kr2}, there exist an $\CN$-valued inner product ${}_\CN \braket{\cdot}{\cdot}$ and an $\CM$-valued inner product $\braket{\cdot}{\cdot}_{\CM}$ on $X^{\bdd}$ defined by the formulas
\begin{align}\label{equ:l_r_inner_prod_def}
  \braket{\xi}{b\beta}_X = \tau_\CN(b \, {}_\CN\braket{\beta}{\xi}), \quad \braket{\xi}{\beta a}_X = \tau_\CM(\braket{\xi}{\beta}_\CM \, a), \quad \forall a \in \CM, b \in \CN
\end{align}
where $\xi,\beta \in X^{\bdd}$ and $\braket{\cdot}{\cdot}_X$ is the inner product of $X$ which is linear in the second variable. If $t: X_1 \to X_2$ is an intertwiner of $\CN$-$\CM$-bimodules, then $t\xi \in X_2^{\bdd}$ for every $\xi \in X_1^{\bdd}$, and
\begin{align*}
 {}_{\CN}\braket{\beta}{t\xi} = {}_{\CN}\braket{t^*\beta}{\xi}, \quad 
 \braket{t\xi}{\beta}_{\CM}= \braket{\xi}{t^*\beta}_{\CM}, \quad \forall \xi \in X_1^{\bdd}, \beta \in X_2^{\bdd}.
\end{align*}

\begin{rem}
With the inner products ${}_\CN \braket{\cdot}{\cdot}$ and $\braket{\cdot}{\cdot}_{\CM}$, the vector space $X^{\bdd}$ is a Hilbert $\CN$-$\CM$-bimodule of \textit{finite type}, see Definition 1.3 and 1.9 \cite{ty}, and Proposition 1.12 and 1.18 in \cite{ty} or Theorem 9.12 in \cite{ek} for the proof of this fact. 
\end{rem}

For a finite $\CM$-$\CL$-bimodule $X$ and a finite $\CN$-$\CM$-bimodule $Y$, let $Y^{\bdd} \odot X^{\bdd}$ be the algebraic tensor product of $Y^{\bdd}$ and $X^{\bdd}$. Consider the positive sesquilinear form on $Y^{\bdd} \odot X^{\bdd}$ defined by 
\begin{align*}
  \braket{\xi_1 \odot \beta_1}{\xi_2 \odot \beta_2} &\dfeq \tau_{\CN}({}_{\CN}\braket{\xi_2 \, {}_{\CM}\braket{\beta_2}{\beta_1}}{\xi_1})\\
  &= \tau_{\CM}(\braket{\xi_1}{\xi_2}_{\CM} {}_{\CM}\braket{\beta_2}{\beta_1})\\
  &= \tau_{\CL}(\braket{\beta_1}{\braket{\xi_1}{\xi_2}_{\CM} \, \beta_2}_{\CL}). 
\end{align*}
Let $\CI$ be the subspace of $Y^{\bdd} \odot X^{\bdd}$ consisting of null vectors $\zeta$, i.e., $\braket{\zeta}{\zeta} = 0$. Then the \textbf{relative tensor product} (or Connes' fusion) of $Y$ and $X$ with respect to $\tau_{\CM}$ is the $\CN$-$\CL$-bimodule $Y \rt{\CM} X$ obtained by completing of $(Y^{\bdd} \odot X^{\bdd})/\CI$ with respect to the inner product defined above. See Section IX.3 in \cite{tkII} or Section 9.7 in \cite{ek} for the basic properties of the relative tensor product. If $Y = \oplus_{kj} Y_{kj}$ and $X = \oplus_{ji} X_{ji}$, then  
\begin{align*}
  Y \rt{\CM} X = \underset{ki}{\scalebox{1.5}{$\oplus$}} \scalebox{1.2}{$($}\oplus_{j} (Y_{kj} \rt{\CM_j} X_{ji})\scalebox{1.2}{$)$} 
\end{align*}
because $\xi a \rt{\CM} \beta = \xi \rt{\CM} a\beta$ in $Y \rt{\CM} X$ for $\xi \in Y$, $\beta \in X$, $a \in \CM$. By Corollary 9.64 in \cite{ek}, $Y \rt{\CM} X$ is a finite $\CN$-$\CL$-bimodule. 

For every finite $\CN$-$\CM$-bimodule $X$, the \textbf{dual} of $X$ is given by the conjugate $\CM$-$\CN$-bimodule $\overline{X}$. Namely, $\overline{X}$ is the conjugate Hilbert space with left and right actions given by $a \overline{\xi} b \dfeq \overline{b^* \xi a^*}$. It is clear that $\overline{X}^{\,\bdd} = \overline{X^{\bdd}}$ and 
\begin{align*}
  {}_{\CM}\braket{\overline{\xi}}{\overline{\beta}} = \braket{\xi}{\beta}_{\CM}, \quad 
  \braket{\overline{\beta}}{\overline{\xi}}_{\CN} = {}_{\CN}\braket{\beta}{\xi}, \quad \forall \xi, \beta \in X.
\end{align*}

\begin{notat}\label{not:vN}
It is known that finite direct sums of II$_1$ factors, finite bimodules and intertwiners constitute a rigid C$^*$-bicategory with finite-dimensional centers in the sense of Section \ref{sec:rigid_Cstar_bicats}. We denote it by $\vN$. Recall that we denoted by $\vN(\CM, \CN)$ the C$^*$-category of $\CN$-$\CM$-bimodules.
\end{notat}

The categorical condition expressed in Definition \ref{def:fin_dim_centers} coincides indeed with the finite-dimensionality of the centers of the von Neumann algebras. The horizontal composition of 1-morphisms in $\vN$ is the relative tensor product of bimodules and the horizontal unit 1-morphism in $\vN(\CM,\CM)$ is the $\CM$-$\CM$-bimodule $L^2(\CM, \tau_\CM)$.
In the following, we regard $\CM$ as a subspace of $L^2(\CM, \tau_\CM)$. Moreover, we identify $(\xi \rt{\CN} \beta) \rt{\CM} \eta$, $b \rt{\CN} \beta$, $\beta \rt{\CM} a$ with $\xi \rt{\CN} (\beta \rt{\CM}\eta)$, $b \beta$, $\beta a$, respectively, see Proposition 3.19 and Theorem 3.20 in \cite{tkII}.

\subsection{C$^*$-Frobenius algebras in $\vN$ and extensions}\label{sec:qsystemsinvN}\,\\ \vspace{-2mm}

Let $\CM$ be a factor and $\End(\CM)$ the tensor C$^*$-category of all the endomorphisms of $\CM$. When $\CM$ is of type III, it is known from \cite{longh} that C$^*$-Frobenius algebras in $\End(\CM)$, also called \textit{Q-systems}, describe the unital not necessarily factorial extensions of $\CM$, see also Theorem 3.11 in \cite{bklr}. We now show the analogous statement for C$^*$-Frobenius algebras in $\vN(\CM,\CM)$, where $\CM$ is a finite direct sum of II$_1$ factors. 

Let 
\begin{align*}
(H, m, \iota) = (H, m:H \rt{\CM} H \to H, \iota:L^2(\CM, \tau_\CM) \to H)
\end{align*} 
be a non-zero C$^*$-Frobenius algebra in $\vN(\CM,\CM)$. Our goal is to show that $H$ is unitarily isomorphic to a C$^*$-Frobenius algebra of the form $\overline{X}\rt{\CN} X$, where $\CN$ is a type II$_1$ von Neumann algebra which extends $\CM$, and $X$ is a finite $\CN$-$\CM$-bimodule.

Since $H$ is a finite $\CM$-$\CM$-bimodule and the center of $\CM$ is finite-dimensional, there exist a left basis $\{\xi_1, \ldots, \xi_l\}$ and a right basis $\{\beta_1, \ldots, \beta_r\}$ of $H^{\bdd}$ such that
\begin{align*}
  \xi = \sum_{i} {}_\CM \braket{\xi}{\xi_i}\xi_i, \quad \beta 
  = \sum_j \beta_j \braket{\beta_j}{\beta}_\CM, \quad \forall \xi, \beta \in H^{\bdd}
\end{align*}
see Theorem 9.12 in \cite{ek}. See also Proposition 1 in \cite{f} and Section 2 in \cite{fk}.

Let $\gamma: L^2(\CM, \tau_\CM) \to \overline{H} \rt{\CM} H$ and $\overline{\gamma}: L^2(\CM, \tau_\CM) \to  H\rt{\CM}\overline{H}$ be the $\CM$-$\CM$-bimodule maps defined by
\begin{align}\label{equ:r_def}
  \gamma(a) := \sum_i \overline{\xi_i} \rt{\CM} (\xi_i a), \quad
  \overline{\gamma}(a) := \sum_i (a\beta_i) \rt{\CM} \overline{\beta_i}, \quad \forall a \in \CM.
\end{align}
By Section 4 in \cite{ty}, $(\gamma, \overline{\gamma})$ is a solution of the conjugate equations \mref{equ:conjeq} for $H$ and $\overline{H}$. Moreover 
\begin{align*}
  \gamma^*(\overline{\xi} \rt{\CM} \beta) = \braket{\xi}{\beta}_\CM, \quad
  \overline{\gamma}^{\, *}(\beta \rt{\CM} \overline{\xi}) = {}_\CM \braket{\beta}{\xi},
  \quad \forall \xi, \beta \in H^{\bdd}.
\end{align*}

As in the factor case, due to the C$^*$-Frobenius algebra structure on $(H, m, \iota)$, the set of bounded vectors $H^{\bdd}$ is a von Neumann algebra acting on $H$, see also Section 2.3 in \cite{wy1}. Indeed
\begin{align*}
  \|m(\xi \rt{\CM} \beta)\|^2 \leq \|m\|^2\|\xi \rt{\CM} \beta\|^2
  =\|m\|^2\braket{\beta}{\braket{\xi}{\xi}_{\CM}\beta} \leq 
  \|m\|^2\|\braket{\xi}{\xi}_{\CM}\|\|\beta\|^2, \quad \forall \xi, \beta \in H^{\bdd}.
\end{align*}
Therefore, the left multiplication by $\xi$, i.e., $\beta\in H \mapsto m(\xi \rt{\CM} \beta)$, extends to a bounded operator $L_\xi: H \to H$. Moreover, $L_\xi$ commutes with the right $\CM$-action since 
\begin{align}\label{equ:N_comm_right_act}
  L_\xi(\beta  a)= m(\xi \rt{\CM} (\beta a)) = L_\xi(\beta) a, \quad \forall \beta \in H^{\bdd}, a \in \CM.
\end{align}

\begin{lem}\label{lem:N_SOT_closed}
Let $\CN \dfeq \{L_\xi: \xi \in H^{\bdd}\}$. Then $\CN$ is a von Neumann algebra and $a \to L_{\iota(a)}$ is a $*$-homomorphism from $\CM$ into $\CN$. If the left action of $\CM$ on $H$ is faithful, then $a \to L_{\iota(a)}$ is injective. 
\end{lem}

\begin{proof}
By the associativity of the multiplication in $H$, namely $m(1_H \rt{\CM} m) = m(m \rt{\CM} 1_H)$ where $1_H$ is the identity map on $H$, we get $L_{\xi_1} L_{\xi_2} = L_{m(\xi_1 \rt{\CM} \xi_2)}$. Similarly, the unit property $m(\iota \rt{\CM}1_H) = 1_H$ implies that $L_{\iota(I)} = 1_H$, where $I$ is the unit of $\CM$. Therefore, $\CN$ is a unital algebra of bounded operators on $H$. Assume that $\{L_{\xi_\alpha}\}$ is a net converging to $t$ in the strong operator topology. Since $L_{\xi_\alpha}(\iota(I)) = \xi_\alpha$, then $\xi_{\alpha}$ converges to a vector $\xi := t(\iota(I))$. Note that $t$ commutes with the right $\CM$ action. We have 
\begin{align*}
  \|\xi a\| = \|t(\iota(a))\| \leq \|t \iota\|\|a\|_2, \quad \forall a \in \CM.
\end{align*}
Thus $\xi \in H^{\bdd}$ and 
\begin{align*}
  t\beta = \lim_{\alpha} m(\xi_\alpha \rt{\CM} \beta) 
  = m(\xi \rt{\CM} \beta) = L_\xi(\beta), \quad \forall \beta \in H^{\bdd}.
\end{align*}
Therefore, $\CN$ is strongly closed.

We now show that $L_\xi^* \in \CN$ for every $\xi \in H^{\bdd}$. Let $\zeta \in H^{\bdd}$. Note that $m^*(\zeta) \in (H \rt{\CM} H)^{\bdd}$. Then Proposition 9.62 in \cite{ek} implies that $m^*(\zeta) = \underset{k=1}{\overset{n}{\sum}} \zeta^{1}_k \rt{\CM} \zeta^{2}_k$, where $\zeta_k^1$, $\zeta^{2}_k \in H^{\bdd}$. Since $(m^* \iota, m^*\iota)$ is a solution of the conjugate equations for $H$ and $H$, by Lemma 8.12 in \cite{lr} there is an invertible $\CM$-$\CM$-bimodule map $f: \overline{H} \to H$ such that
\begin{align}\label{equ:can_solution_to_Fro_sol}
    (f \rt{\CM} 1_H)\gamma = m^*\iota = (1_H \rt{\CM} (f^{-1})^{*})\overline{\gamma}
\end{align}
where $\gamma$ and $\overline{\gamma}$ are defined by (\ref{equ:r_def}). Then we have
\begin{align*}
    \braket{L_\xi(\beta)}{\zeta} &= \braket{m(\xi \rt{\CM} \beta)}{\zeta} 
    =\sum_k \braket{\xi \rt{\CM} \beta}{\zeta^{1}_k \rt{\CM} \zeta^2_k}
    =\sum_k \braket{\beta}{\braket{\xi}{\zeta^{1}_k}_{\CM} \zeta^2_k}\\
    &=\sum_k \braket{\beta}{(\iota^*m \rt{\CM} 1_H)
    [(f^{-1})^{*}(\overline{\xi})\rt{\CM} \zeta^{1}_k \rt{\CM} \zeta^2_k]}
    =\braket{\beta}{m[(f^*)^{-1}(\overline{\xi}) \rt{\CM} \zeta]}
\end{align*}
where we use the fact $\gamma = (f^{-1} \rt{\CM} 1_H)m^* \iota$ in the fourth equality and the Frobenius property in the last one. Therefore, $L^{*}_{\xi} = L_{(f^*)^{-1}(\overline{\xi})} \in \CN$, and $\CN$ is a von Neumann algebra. 

Finally, note that $L_{\iota(a)}(\xi) = a\xi$ for every $a \in \CM$, $\xi\in H$, thus $a \to L_{\iota(a)}$ is a $*$-homomorphism, and the proof is complete.
\end{proof}

In order to show that the von Neumann algebra $\CN$ defined in Lemma \ref{lem:N_SOT_closed} is a finite direct sum of II$_1$ factors, we first need to characterize the center of $\CN$. Recall that $H$ is an $H$-$H$-bimodule in $\vN(\CM, \CM)$. If $\CM$ is a factor, then it is known that $\CN \cap \CN' = \BMod_\vN(H,H)(H,H)$, the space of $H$-$H$-bimodule intertwiners between $H$ and itself, see Lemma 3.17 in \cite{bklr}. It is straightforward to extend this result to the case when $\CM$ is not a factor and we briefly unravel the proof for the reader's convenience.  

\begin{lem}\label{lem:commutant_desc}
Let $t: H \to H$ be a bounded operator. Then $t \in \{L_{\iota(a)}: a \in \CM\}' \cap \CN$ if and only if $t$ is an $\CM$-$\CM$-bimodule intertwiner and $tm = m(t \rt{\CM} 1_H)$. Furthermore, $t \in \CN' \cap \CN$ if and only if $t$ is an $\CM$-$\CM$-bimodule intertwiner and $t m = m(t \rt{\CM} 1_H)=m(1_H \rt{\CM} t)$.
\end{lem}

\begin{proof} 
If $t \in \CN$ and $tL_{\iota(a)} = L_{\iota(a)}t$ for every $a \in \CM$, then $t$ is an $\CM$-$\CM$-bimodule intertwiner and $tm = m(t \rt{\CM} 1_H)$ by equation \mref{equ:N_comm_right_act} and by associativity $m(m \rt{\CM} 1_H) = m(1_H \rt{\CM} m)$. Conversely, $t\beta = tm(\iota(I) \rt{\CM} \beta) = m(\xi \rt{\CM} \beta) = L_\xi(\beta)$ for every $\beta \in H^{\bdd}$, where $\xi = t\iota(I) \in H^{\bdd}$. Note that $L_\xi L_\beta(\zeta)=L_\xi m (\beta \rt{\CM} \zeta)$ and $L_\beta L_\xi (\zeta)= m(\beta \rt{\CM} L_{\xi}(\zeta))$. Thus $\CN \cap \CN' = \BMod_\vN(H,H)(H,H)$. 
\end{proof}

\begin{thm}\label{thm:CN_is_finite_sum_II_factors}
The von Neumann algebra $\CN$ defined in Lemma \ref{lem:N_SOT_closed} is a finite direct sum of II$_1$ factors.
\end{thm}

\begin{proof}
Since $H$ is a finite $\CM$-$\CM$-bimodule, the algebra $\CA$ of operators commuting with the right $\CM$-action on $H$ is a finite von Neumann algebra. Note that $\CN \subset \CA$. Then $\CN$ is a finite von Neumann algebra. By Lemma \ref{lem:commutant_desc}, the center of $\CN$ equals $\BMod_\vN(H,H)(H,H)$, thus it is finite-dimensional and $\CN$ is a finite direct sum of II$_1$ factors. 
\end{proof}

\begin{lem}\label{lem:trace_K_def}
There is an invertible positive operator $k$ in $\{L_{\iota(a)}: a \in \CM\}' \cap \CN$ such that 
\begin{align*}
  \tau_\CN(k^2 b) = \braket{\iota(I)}{b \iota(I)}, \quad \forall b \in \CN
\end{align*}
where $I$ is the identity of $\CM$.
\end{lem}

\begin{proof}
Since $H$ is a finite $\CM$-$\CM$-bimodule, $\iota(I)$ is an $\CN$-bounded vector, see the proof of Proposition 9.57 in \cite{ek}. Then the Sakai-Randon-Nikdo\'{y}m theorem implies that there is a positive operator $k$ in $\CN$ satisfying $\tau_\CN(k^2 b) = \braket{\iota(I)}{b \iota(I)}$ for every $b\in\CN$. Note that $\tau_\CN(k^2L_{\iota(a)}b) = \braket{\iota(I)}{ab \iota(I)} = \braket{\iota(I)  a^*}{b \iota(I)} = \braket{\iota(I)}{b \iota(a)} = \tau_\CN(L_{\iota(a)}k^2 b)$, for every $a \in \CM$, $b \in \CN$, thus $k L_{\iota(a)} = L_{\iota(a)} k$. Since $L_\xi(\iota(I)) = \xi$, the vector $\iota(I)$ is separating for $\CN$. By Lemma \ref{lem:commutant_desc}, $\{L_{\iota(a)}: a \in \CM\}' \cap \CN$ is finite-dimensional, therefore $k$ is invertible.
\end{proof}

Let $\CN$ be the von Neumann algebra defined in Lemma \ref{lem:N_SOT_closed}. In the following, we use $X$ to denote the $\CN$-$\CM$-bimodule $L^2(\CN, \tau_\CN)$ with the canonical $\CN$-$\CM$-bimodule structure. More explicitly, if we view $\CN$ as a subspace of $L^2(\CN, \tau_\CN)$, the left $\CN$-action and the right $\CM$-action are given by
\begin{align}\label{equ:def_X_actions}
  b_1 \cdot b_2 \cdot a : = b_1 b_2 L_{\iota(a)}, \quad \forall b_1, b_2 \in \CN, a \in \CM. 
\end{align}
Similarly, let $\overline{X}$ be the $\CM$-$\CN$-bimodule $L^2(\CN, \tau_\CN)$ with left $\CM$-action and right $\CN$-action given by 
\begin{align}
  a \cdot b_2 \cdot b_1 : = L_{\iota(a)} b_2 b_1, \quad \forall b_1, b_2 \in \CN, a \in \CM. 
\end{align}
Note that $X^{\bdd} = \CN$, see Example 9.6 and 9.10 in \cite{ek}. Let $k$ be the invertible operator defined in Lemma \ref{lem:trace_K_def}. Since for every $\xi$, $\beta \in H^{\bdd}$ we have
\begin{align}\label{equ:u_def_con}
  \braket{\xi}{\beta} = \braket{\iota(I)}{L_{\xi}^*L_{\beta} \iota(I)} = \tau_{\CN}((L_{\xi}k)^*(L_{\beta}k))
\end{align}
the map $\xi \to L_\xi k$ extends to a unitary from the finite $\CN$-$\CM$-bimodule $H$ to the finite $\CN$-$\CM$-bimodule $X$ intertwining the left $\CN$ and right $\CM$-actions. 

Let $\gamma: L^2(\CM, \tau_\CM) \to \overline{X} \rt{\CN} X$ be the operator defined by 
\begin{align*}
\gamma(a) \dfeq k \rt{\CN} L_{\iota(a)} = L_{\iota(a)} \rt{\CN} k, \quad \forall a \in \CM.
\end{align*} 

Let $\overline{\gamma}: L^2(\CN, \tau_\CN) \to X \rt{\CM} \overline{X}$ be the operator whose adjoint is defined by 
\begin{align*}
\overline{\gamma}^{\, *}(b_1 \rt{\CM} b_2) := b_1 k^{-1} b_2, \quad \forall b_1, b_2 \in \CN.
\end{align*} 

Then $(\gamma, \overline{\gamma})$ is a solution of the conjugate equations \mref{equ:conjeq} for $X$ and $\overline{X}$. Therefore, $\overline{X}$ is a dual of $X$ and $(\overline{X}\rt{\CN} X, 1_{\overline{X}}\, \rt{\CN} \overline{\gamma}^{\, *} \rt{\CN} 1_{X}, \gamma)$ is a C$^*$-Frobenius algebra.

\begin{lem}\label{lem:u_def}
Let $X \in \vN(\CM,\CN)$ and $\overline{X} \in \vN(\CN,\CM)$ be the bimodules defined above.
Then the two C$^*$-Frobenius algebras $(H, m, \iota)$ and $(\overline{X}\rt{\CN} X, 1_{\overline{X}}\, \rt{\CN} \overline{\gamma}^{\, *} \rt{\CN} 1_{X}, \gamma)$ are unitarily isomorphic, namely there is a unitary $u$ in $\vN(\CM, \CM)(H, \overline{X}\rt{\CN} X)$ which is an algebra isomorphism. Furthermore, $H$ is special if and only if $\overline{\gamma}^{\,*}\overline{\gamma} = 1_{L^2(\CN, \tau_\CN)}$, i.e., if and only if $\overline{X}\rt{\CN} X$ is special.
\end{lem}

\begin{proof}
By the definition of the inner product on $\overline{X}\rt{\CN} X$ and by equation \mref{equ:u_def_con}, we know that the map $\xi \in H^{\bdd} \mapsto L_{\xi} \rt{\CN} k$, extends to a unitary $u : H \to \overline{X}\rt{\CN} X$. Since $L_{a \xi b} = L_{\iota(a)}L_{\xi}L_{\iota(b)}$ for every $a, b \in \CM$ and $L_{\iota(b)}k = kL_{\iota(b)}$, $u$ is an $\CM$-$\CM$-bimodule intertwiner. Now, note that for every $a \in \CM$, we have $u \iota(a) = L_{\iota(a)} \rt{\CN} k = \gamma(a)$ and  
\begin{align*}
  (1_{\overline{X}}\, \rt{\CN} \overline{\gamma}^{\, *} \rt{\CN} 1_{X})(u(\xi) \rt{\CM} u(\beta)) 
  &= (1_{\overline{X}}\, \rt{\CN} \overline{\gamma}^{\, *} \rt{\CN} 1_{X})(L_\xi \rt{\CN} k \rt{\CM} L_{\beta} \rt{\CN}k) 
  = (L_\xi L_\beta) \rt{\CN} k \\
  &= L_{m(\xi \rt{\CM} \beta)} \rt{\CN} k = um(\xi \rt{\CM} \beta)
\end{align*}
for every $\xi, \beta \in H^{\bdd}$. Thus $u$ is an algebra isomorphism. 

Finally, note that $\overline{X} \rt{\CN} X \rt{\CM} \overline{X} \rt{\CN} X = \{L_{\iota(I)} \rt{\CN} \zeta \rt{\CN} L_{\iota(I)}: \zeta \in X \rt{\CM} \overline{X}\}$, thus $1_{\overline{X}} \rt{\CN} (\overline{\gamma}^{\, *} \overline{\gamma})\rt{\CN} 1_{X}$ is the identity map if and only if $\overline{\gamma}^{\, *}\overline{\gamma}$ is the identity map. 
\end{proof}

\begin{thm}\label{thm:real_in_multi_tensor}
Let $\CC$ be an indecomposable multitensor C$^*$-category (with arbitrary spectrum). Then there is a finite direct sum of II$_1$ factors $\CN$ and a fully faithful tensor $*$-functor form $\CC$ to $\vN(\CN, \CN)$.
\end{thm}

\begin{proof}
By Corollary \ref{cor:ind_ten_equ_alg_ten}, there is a tensor C$^*$-category $\CC_0$ and a standard C$^*$-Frobenius algebra $(H,m,\iota)$ in $\CC_0$ such that $\CC \simeq \sBMod_{\CC_0}(H,H)$. By Theorem 4.15 in \cite{lw}, we may assume that $\CC_0$ is a full tensor sub-C$^*$-category of $\vN(\CM, \CM)$ where $\CM$ is a II$_1$ factor. In particular, $H$ is a finite $\CM$-$\CM$-bimodule with a C$^*$-Frobenius algebra structure, and $\sBMod_{\CC_0}(H,H)$ is a full subcategory of $\sBMod_{\vN(\CM, \CM)}(H,H)$. 

Let $\CN$ be the von Neumann algebra from Theorem \ref{thm:CN_is_finite_sum_II_factors}. Since $H$ is standard, thus special, by Lemma \ref{lem:u_def} and Lemma \ref{lem:2-cat_std_Frob_bimod_equ} applied to $\bB = \vN$ and $\CN=\CL$, we have
\begin{align*}
    \sBMod_{\vN(\CM, \CM)}(H,H) \simeq \sBMod_{\vN(\CM, \CM)}(\overline{X} \rt{\CN} X,\overline{X} \rt{\CN} X) \simeq \vN(\CN, \CN).
\end{align*}
Thus the proof is complete.
\end{proof}

As a special case, we obtain a realization result for multifusion C$^*$-categories over hyperfinite von Neumann algebras.
The same result, obtained with different techniques, has been announced in \cite{penowr} and will appear in \cite{distortion}.
Recall that a \textbf{multifusion C$^*$-category} is a multitensor C$^*$-category with finitely many isomorphism classes of simple objects, i.e., with finite spectrum.

\begin{cor}
Every indecomposable multifusion C$^*$-category can be realized as finite bimodules over a finite direct sum of the hyperfinite II$_1$ factor.
\end{cor}

\begin{proof}
In the proof of Theorem \ref{thm:real_in_multi_tensor}, we may assume that $\CC_0$ is a fusion C$^*$-category. Thus $\CC_0$ can be realized as finite bimodules over the hyperfinite II$_1$ factor $\CR$ by Theorem 7.6 in \cite{hy}. Let $\CN$ be the von Neumann algebra considered in Theorem \ref{thm:CN_is_finite_sum_II_factors}. Since $L^2(\CN, \tau_\CN)$ is a finite right $\CR$-module, we know that the commutant of the right $\CR$-action is a hyperfinite II$_1$ factor containing $\CN$. Therefore $\CN$ is a finite direct sum of hyperfinite II$_1$ factors, see Corollary 2 in \cite{connes}.
\end{proof}

\begin{rem}
In \cite{distortion}, given a multifusion C$^*$-category, a representation over a hyperfinite von Neumann algebra is constructed using planar algebra/commuting square techniques by first defining an infinite tower of inclusions of finite-dimensional algebras with a coherent family of faithful traces, and then taking inductive limits. Moreover, in the case of multifusion C$^*$-categories, a complete classification result is obtained in \cite{distortion} for such representations, by means of an additional, not purely categorical, invariant which characterizes Popa's homogeneity and completes the standard invariant for finite depth finite index connected inclusions of hyperfinite II$_1$ von Neumann algebras with finite-dimensional centers.
\end{rem}

\section{Realization of rigid C$^*$-bicategories}
\subsection{Proof of Theorem \ref{thm:real_2_cat}}\label{sec:pfmainthm}\,\\ \vspace{-2mm}

This section is devoted to the proof of the main result of this note: every rigid C$^*$-bicategory with finite-dimensional centers $\bB$ can be realized as finite bimodules over finite direct sums of II$_1$ factors. Namely, there is a locally fully faithful $*$-bifunctor from $\bB$ to $\vN$.

From Notation \ref{not:vN}, $\vN$ is a rigid C$^*$-bicategory with finite-dimensional centers. From Remark \ref{rem:BinsBiModB}, $\vN$ can be viewed as a full sub-C$^*$-bicategory of the C$^*$-bicategory $\sBMod_\vN$. The embedding is given by identifying a von Neumann algebra $\CN$ with the trivial special C$^*$-Frobenius algebra $\one_\CN=L^2(\CN,\tau_\CN)$ in $\vN(\CN,\CN)$. 

The following lemma shows that Connes' bimodules in $\vN$ can be identified with the categorical special bimodules over special C$^*$-Frobenius algebras in $\vN$.

\begin{lem}\label{lem:von_bim_alg_closed}
The inclusion $\vN \to \sBMod_\vN$ is a $*$-biequivalence. 
\end{lem}

\begin{proof}
By Lemma 3.1 in \cite{ng1} or Theorem 7.4.1 in \cite{jy} we only need to show that the inclusion is biessentially surjective on objects. Namely, we have to show that for every special C$^*$-Frobenius algebra $H$ in $\vN(\CM, \CM)$ there is a finite direct sum of II$_1$-factors $\CN$ such that the two special C$^*$-Frobenius algebras $L^2(\CN, \tau_\CN)$ and $H$ are equivalent in $\sBMod_\vN$, see Definition 1.6 in \cite{ng1}.

Let $\CN$ be the von Neumann algebra considered in Theorem \ref{thm:CN_is_finite_sum_II_factors} and $X$ be the $\CN$-$\CM$-bimodule defined by equation \mref{equ:def_X_actions}. Since $H$ is a special C$^*$-Frobenius algebra by assumption, $\overline{X}\rt{\CN} X$ is a special C$^*$-Frobenius algebra by Lemma \ref{lem:u_def}. It is clear that $X$ is an $L^2(\CN, \tau_\CN)$-$(\overline{X}\rt{\CN} X)$-bimodule. Then $X$ is an $L^2(\CN, \tau_\CN)$-$H$-bimodule with the right $H$-action induced by the algebra isomorphism of Lemma \ref{lem:u_def}. Similarly, $\overline{X}$ is a $H$-$L^2(\CN, \tau_\CN)$-bimodule. We have to show that $X$ and $\overline{X}$ are each other's inverse.
By Lemma \ref{lem:u_def}, we know that $\overline{X} \rt{\CN} X \simeq \one_H = H$. Moreover, $p^H_{X \otimes \overline{X}} = \overline{\gamma}\, \overline{\gamma}^{\, *}$, thus $X \rt{H} \overline{X} = \one_{L^2(\CN, \tau_\CN)} = L^2(\CN, \tau_\CN)$. This proves the statement. 
\end{proof}

\begin{rem}
After finishing this article, we noticed that the notion of idempotent completion of a locally idempotent complete bicategory has been introduced in \cite{cdr}. A locally idempotent complete bicategory is called idempotent complete if it is biequivalent to its idempotent completion. It is not hard to see that every rigid C$^*$-bicategory is locally idempotent complete, and its idempotent completion is the bicategory of special C$^*$-Frobenius algebras and special bimodules defined in Proposition \ref{prop:special_Fro_alg_bicat}. By Lemma \ref{lem:von_bim_alg_closed}, we know that the bicategory $\vN$ is idempotent complete. 
\end{rem}

We are now ready to prove Theorem \ref{thm:real_2_cat}.
Let $\bB$ be a rigid C$^*$-bicategory with finite-dimensional centers. 
For every object $\CN$, consider the decomposition
\begin{align*}
\bB(\CN,\CN) = \oplus_{j} \CN_j
\end{align*}
of the multitensor C$^*$-category $\bB(\CN, \CN)$ as finite direct sum of \textit{indecomposable} multitensor C$^*$-categories, denoted by $\CN_j$.
The tensor unit $\one_{\CN_j}$ of each $\CN_j$ is a standard (thus special) C$^*$-Frobenius algebra in $\bB(\CN,\CN)$, since $\one_{\CN_i} \otimes \one_{\CN_j} = \delta_{ij} \one_{\CN_j}$. Moreover, $\one_{\CN} = \oplus_j\one_{\CN_j}$ is a decomposition of the tensor unit $\one_\CN$ of $\bB(\CN,\CN)$. Each $\one_{\CN_j}$ is a maximal direct summand of $\one_\CN$ such that $\one_{\CN_j} \otimes \bB(\CN,\CN) \otimes \one_{\CN_j}$ is indecomposable.  
The idea here is to \lq\lq decompose" each object $\CN$ in $\bB$ into \lq\lq indecomposable subobjects" represented by the standard C$^*$-Frobenius algebras $\one_{\CN_j}$, viewing $\bB$ as a sub-C$^*$-bicategory of $\sBMod_\bB$. A similar idea of decomposing objects 
appears in \cite{z} at the end of Section 1, but to our knowledge it hasn't been exploited thereafter.

\begin{notat}\label{not:set_lambda}
Let $\Lambda$ be the union of all these standard C$^*$-Frobenius algebras $\one_{\CN_j}$, varying $\CN$ in $\bB$.
\end{notat}

\begin{lem}\label{lem:dec_bB}
With the above notation, we introduce a binary relation $\sim$ on $\Lambda$ by defining $\one_{\CL_i} \sim \one_{\CM_j}$ whenever $\one_{\CL_i} \otimes \bB(\CM, \CL) \otimes \one_{\CM_j} \neq 0$. Then $\sim$ is an equivalence relation.    
\end{lem}

\begin{proof}
It is clear that $\sim$ is reflexive. It is symmetric because $\overline{X} \in \one_{\CM_j} \otimes \bB(\CL, \CM) \otimes \one_{\CL_i}$ for every $X \in \one_{\CL_i} \otimes \bB(\CM, \CL) \otimes \one_{\CM_j}$.  Now, assume that $\one_{\CL_i} \sim \one_{\CM_j}$ and $\one_{\CM_j} \sim \one_{\CN_k}$, and choose non-zero 1-morphisms $X \in \one_{\CL_i} \otimes \bB(\CM, \CL) \otimes \one_{\CM_j}$ and $Y \in \one_{\CM_j} \otimes \bB(\CN, \CM) \otimes \one_{\CN_k}$. Let $e$ and $f$ be two simple sub-1-morphisms of $\one_{\CM_j}$ such that $X \otimes e$ and $f \otimes Y$ are non-zero. Then for every non-zero $Z \in e \otimes \bB(\CM, \CM) \otimes f$, we have $X \otimes Z \otimes Y \neq 0$. This shows that $\sim$ is transitive.
\end{proof}

\begin{rem}\label{rem:maximalitynsim}
Note that $\one_{\CM_i}$ and $\one_{\CM_j}$ are not equivalent, unless $i=j$, because each $\one_{\CM_i}$ is a maximal direct summand of $\one_\CM$ such that $\one_{\CM_i} \otimes \bB(\CM,\CM) \otimes \one_{\CM_i}$ is indecomposable. Thus, for every $\one_{\CM_j}$ and every $\CL$ in $\bB$, there exists at most one $\one_{\CL_i}$ such that $\one_{\CL_i} \sim \one_{\CM_i}$. 
\end{rem}

\begin{lem}\label{lem:connected_real}
    If $\bB(\CM, \CN) \neq 0$ and $\bB(\CM, \CM)$, $\bB(\CN, \CN)$ are indecomposable for every pair of objects $\CM$, $\CN$ in $\bB$, then there exists a locally fully faithful $*$-bifunctor from $\bB$ to $\vN$. 
\end{lem}

\begin{proof}
By Lemma \ref{lem:indecom_f_alg} and \ref{lem:2-cat_std_Frob_bimod_equ}, there is a locally fully faithful $*$-bifunctor from $\bB$ to $\sBMod_{\bB(\CM, \CM)}$ where $\CM$ is a fixed object in $\bB$. By Theorem \ref{thm:real_in_multi_tensor}, $\bB(\CM, \CM)$ can be identified with a full subcategory of $\vN(\CA,\CA)$ where $\CA$ is a finite direct sum of II$_1$ factors, and Lemma \ref{lem:von_bim_alg_closed} implies the result. 
\end{proof}

\begin{proof}[Proof of Theorem \ref{thm:real_2_cat}]
For a rigid C$^*$-bicategory with finite-dimensional centers $\bB$, the set $\Lambda$ defined above splits into equivalence classes, denoted by $\Lambda_k$, $k\in\Lambda/\!\!\sim$, with respect to the equivalence relation introduced in Lemma \ref{lem:dec_bB}. 
We identify $\bB$ with its image in $\sBMod_\bB$.
Let $\bB_k$ be the full sub-C$^*$-bicategory of $\sBMod_\bB$ such that the set of objects is $\Lambda_k$. 
Since $\bB_k$ satisfies the conditions in Lemma \ref{lem:connected_real}, there is a locally fully faithful $*$-bifunctor $\Phi_k$ from $\bB_k$ to $\vN$. We can define a locally fully faithful $*$-bifunctor $\Phi$ from $\bB$ to $\vN$ as follows
\begin{enumerate}
    \item For every $\CN$ in $\bB$, let $\Phi(\CN) \dfeq \oplus_{k \in \Lambda/\!\!\sim} \Phi(\CN)^{(k)}$, where 
\begin{align*}
    \Phi(\CN)^{(k)} \dfeq \begin{cases}
        \Phi_{k}(\one_{\CN_i}) & \mbox{if $\one_{\CN_i} \in \Lambda_k$ for some $i$},\\
        0 & \mbox{otherwise}.
    \end{cases}
\end{align*}

\item For every $X \in \bB(\CM, \CN)$, let $\Phi(X) := \oplus_{k \in \Lambda/\!\!\sim} X^{(k)}$, where  
    \begin{align*}
       X^{(k)} \dfeq \begin{cases}
        \Phi_{k}(\one_{\CN_i} \otimes X \otimes \one_{\CM_j}) & \mbox{if $\one_{\CN_i}$, $\one_{\CM_j} \in \Lambda_k$ for some $i,j$},\\
        0 & \mbox{otherwise}.
    \end{cases} 
    \end{align*}
        It is clear that $\Phi(X)$ is a finite $\Phi(\CN)$-$\Phi(\CM)$-bimodule since every $X^{(k)}$ is a $\Phi_k (\one_{\CN_i})$-$\Phi_k (\one_{\CM_j})$-bimodule. 
        Similarly, for every $f\in\bB(\CM,\CN)(X_1,X_2)$, let $\Phi(f) := \oplus_{k \in \Lambda/\!\!\sim} f^{(k)}$, where  
    \begin{align*}
       f^{(k)} \dfeq \begin{cases}
        \Phi_{k}(\one_{\CN_i} \otimes f \otimes \one_{\CM_j}) & \mbox{if $\one_{\CN_i}$, $\one_{\CM_j} \in \Lambda_k$ for some $i,j$},\\
        0 & \mbox{otherwise}.
    \end{cases} 
    \end{align*}
\end{enumerate}    
The tensorator and unitor of $\Phi$ are induced from those of the $\Phi_k$'s. 

    Let $X \in \bB(\CM, \CN)$. If $\one_{\CN_i} \otimes X$ is non-zero, then there exists a unique $\one_{\CM_j}$ such that $\one_{\CN_i} \otimes X = \one_{\CN_i} \otimes X \otimes \one_{\CM_j}$, see Remark \ref{rem:maximalitynsim}. Note that $X = \bigoplus_{(\one_{\CN_i}, \one_{\CM_j})} \one_{\CN_i} \otimes X \otimes \one_{\CM_j}$, where the sum extends over all pairs $(\one_{\CN_i}, \one_{\CM_j})$ such that $\one_{\CN_i} \sim \one_{\CM_j}$. 
    Then it is clear that $\Phi|_{\bB(\CM, \CN)}$ is a $*$-functor. Let $Y \in \bB(\CL, \CM)$. By the definition of the relative tensor product of bimodules, we have
    \begin{align*}
        \Phi(X) \otimes_{\Phi(\CM)} \Phi(Y) = \bigoplus_{l \in \Lambda/\!\!\sim} X^{(l)} \otimes_{\Phi(\CM)^{(l)}} Y^{(l)}.
    \end{align*}
Note that $X \otimes Y = \bigoplus_{(\one_{\CN_i}, \one_{\CM_j}, \one_{\CL_k})} \one_{\CN_i} \otimes X \otimes \one_{\CM_j} \otimes Y \otimes \one_{\CL_k}$, where the sum extends over all triples $(\one_{\CN_i}, \one_{\CM_j}, \one_{\CL_k})$ such that $\one_{\CN_i} \sim \one_{\CM_j} \sim \one_{\CL_k}$. By the $*$-bifunctoriality of each $\Phi_l$, we have that $\Phi(X) \otimes_{\Phi(\CM)} \Phi(Y) \simeq \Phi(X \otimes Y)$. Now it is not hard to check that $\Phi$ is a $*$-bifunctor.  
\end{proof}


\bigskip

\noindent\textbf{Acknowledgements.}
We thank Marcel Bischoff for comments on an earlier version of the manuscript. We also want to thank the anonymous referee for constructive criticisms and valuable comments. L.G. also acknowledges the \lq\lq MIUR Excellence Department Project" awarded to the Department of Mathematics, University of Rome Tor Vergata, CUP E83C18000100006, and the \lq\lq MIUR Excellence Department Project MatMod@TOV" awarded to the Department of Mathematics, University of Rome Tor Vergata, CUP E83C23000330006.

\end{document}